\documentclass[12pt]{article}

\usepackage{etex}
\usepackage[top=1.2in, bottom=1.2in, left=1in, right=1in]{geometry}

\usepackage{amsmath}
\usepackage{amssymb}
\usepackage{amsthm}
\usepackage{array}
\usepackage{bbm}
\usepackage{bm}
\usepackage{cancel}
\usepackage{cmap}
\usepackage{csquotes}
\usepackage{enumerate}
\usepackage{enumitem}%resume lists
\usepackage{esint}%allows \ointclockwise, \ointctrclockwise, \var...
\usepackage{fancyhdr}
\usepackage{listings}
\usepackage{longtable}
\usepackage{makeidx}
\usepackage{mathdots}%iddots: dots going northeast
\usepackage{mathtools}
\usepackage{mathrsfs}
\usepackage{stackrel}
\usepackage{stmaryrd}%\mapsfrom
\usepackage{tabularx}
\usepackage{tikz}
\usepackage{ctable}
\usepackage{titlesec}
\usepackage{titletoc}
\usepackage{url}
\usepackage{verbatim}
\usepackage{wasysym}
\usepackage{wrapfig}
\usepackage{yhmath}
\usepackage[all,cmtip]{xy}%Commutative diagrams
\usepackage{hyperref}
\usepackage{cleveref}
\usepackage{algorithm}
\usepackage{algpseudocode}
\algnewcommand\algorithmicinput{\textbf{Input: }}
\algnewcommand\INPUT{\State\algorithmicinput}
\algnewcommand\algorithmicinitialize{\textbf{Initialize: }}
\algnewcommand\INIT{\State\algorithmicinitialize}
\algnewcommand\algorithmicrun{\textbf{Run: }}
\algnewcommand\RUN{\State\algorithmicrun}
\algnewcommand\algorithmicupdate{\textbf{Update: }}
\algnewcommand\UPDATE{\State\algorithmicupdate}
\algnewcommand\algorithmicset{\textbf{Set: }}
\algnewcommand\SET{\State\algorithmicset}
\algnewcommand\algorithmicquery{\textbf{Query: }}
\algnewcommand\QUERY{\State\algorithmicquery}
\algnewcommand\algorithmicoutput{\textbf{Output: }}
\algnewcommand\OUTPUT{\State\algorithmicoutput}
\usetikzlibrary{calc,trees,positioning,arrows,chains,shapes.geometric,%
    decorations.pathreplacing,decorations.pathmorphing,shapes,%
    matrix,shapes.symbols,shadows,fadings}

\usepackage[refpage]{nomencl}

\makenomenclature
\makeindex

\newtheoremstyle{norm}
{12pt}
{12pt}
{}
{}
{\bf}
{:}
{.5em}
{}

\newtheorem{thm}{Theorem}[section]
\newtheorem*{thm*}{Theorem}

\newtheorem*{clm*}{Claim}

\newtheorem*{conj*}{Conjecture}

\newtheorem*{lem*}{Lemma}

\theoremstyle{norm}
\newtheorem{prb}[thm]{Problem}%[section]
\newtheorem*{prb*}{Problem}

\newtheorem*{ax*}{Axiom}

\newtheorem*{df*}{Definition}

\newtheorem*{ex*}{Example}

\newtheorem{expl}[thm]{Exploration}%prb

\newtheorem*{pos*}{Postulate}

\newtheorem*{pr*}{Proposition}

\newtheorem*{qu*}{Question}
\newtheorem{rem}[thm]{Remark}
\newtheorem*{rem*}{Remark}

\usepackage[english]{babel} 
\usepackage{blindtext} 
\newcommand{\cB}[0]{\mathcal{B}}

\newcommand{\cC}[0]{\mathcal{C}}

\newcommand{\E}[0]{\mathbb{E}}

\newcommand{\Pj}[0]{\mathbb{P}}

\newcommand{\R}[0]{\mathbb{R}}

\newcommand{\one}[0]{\mathbf{1}}
\newcommand{\ga}[0]{\gamma}

\newcommand{\de}[0]{\delta}

\newcommand{\ep}[0]{\varepsilon}

\newcommand{\rh}[0]{\rho}

\newcommand{\Te}[0]{\Theta}

\newcommand{\Om}[0]{\Omega}
\newcommand{\si}[0]{\sigma}

\newcommand{\ze}[0]{\zeta}

\newcommand{\nin}[0]{\not\in}
\newcommand{\opl}[0]{\oplus}

\newcommand{\sub}[0]{\subset}

\newcommand{\subeq}[0]{\subseteq}
\newcommand{\supeq}[0]{\supseteq}

\newcommand{\bs}[0]{\backslash}
\newcommand{\iy}[0]{\infty}
\newcommand{\rc}[1]{\frac{1}{#1}}
\newcommand{\prc}[1]{\pa{\rc{#1}}}
\newcommand{\ff}[2]{\left\lfloor\frac{#1}{#2}\right\rfloor}

\newcommand{\fc}[2]{\frac{#1}{#2}}

\newcommand{\pf}[2]{\pa{\frac{#1}{#2}}}%Shortcut for fraction with parentheses

\newcommand{\ba}[1]{\left[ {#1} \right]}

\newcommand{\ce}[1]{\left\lceil {#1}\right\rceil}
\newcommand{\fl}[1]{\left\lfloor {#1}\right\rfloor}

\newcommand{\pa}[1]{\left( {#1} \right)}

\newcommand{\set}[2]{\left\{{#1}:{#2}\right\}}

\newcommand{\wt}[1]{\widetilde{#1}}

\newcommand{\poly}{\operatorname{poly}}

\providecommand{\cal}[1]{\mathcal{#1}}
\renewcommand{\cal}[1]{\mathcal{#1}}

\newcommand{\pull}[9]{
#1\ar@/_/[ddr]_{#2} \ar@{.>}[rd]^{#3} \ar@/^/[rrd]^{#4} & &\\
& #5\ar[r]^{#6}\ar[d]^{#8} &#7\ar[d]^{#9} \\}

\newcommand{\cmp}[9]{
\xymatrix{
#1 \ar[r]^{#4}{#5} \ar@/_2pc/[rr]^{#8}_{#9} & #2 \ar[r]^{#6}_{#7} & #3
}
}

\newcommand{\ha}[1]{\ar@{^(->}[#1]}
\newcommand{\ls}[1]{\ar@{-}[#1]}
\newcommand{\sj}[1]{\ar@{->>}[#1]}
\newcommand{\aq}[1]{\ar@{=}[#1]}
\newcommand{\acir}[1]{\ar@{}[#1]|-{\textstyle{\circlearrowright}}}
\newcommand{\acil}[1]{\ar@{}[#1]|-{\textstyle{\circlearrowleft}}}
\newcommand{\ard}[1]{\ar@{.>}[#1]}
\newcommand{\mt}[1]{\ar@{|->}[#1]}
\newcommand{\inm}[1]{\ar@{}[#1]|-{\in}}
\newcommand{\inr}{\ar@{}[d]|-{\rotatebox[origin=c]{-90}{$\in$}}}
\newcommand{\inl}{\ar@{}[u]|-{\rotatebox[origin=c]{90}{$\in$}}}

\newcommand{\sumo}[2]{\sum_{#1=1}^{#2}}

\newcommand{\prodo}[2]{\prod_{#1=1}^{#2}}

\newcommand{\beq}[1]{\begin{equation}\llabel{#1}}
\newcommand{\eeq}[0]{\end{equation}}
\newcommand{\bal}[0]{\begin{align*}}
\newcommand{\eal}[0]{\end{align*}}%this doesn't work; i don't know why
\newcommand{\ban}[0]{\begin{align}}
\newcommand{\ean}[0]{\end{align}}

\newcommand{\fixme}[1]{{\color{red}#1}}
\newcommand{\llabel}[1]{\label{#1}\text{\fixme{\tiny#1}}}

\newcommand{\vocab}[1]{\emph{#1}} %also index automatically.

\allowdisplaybreaks[2]

\DeclareFontFamily{U}{wncy}{}
    \DeclareFontShape{U}{wncy}{m}{n}{<->wncyr10}{}
    \DeclareSymbolFont{mcy}{U}{wncy}{m}{n}
    \DeclareMathSymbol{\Sh}{\mathord}{mcy}{"58} 

\usepackage{thm-restate}

\newcommand{\M}{\mathcal{M}}
\newcommand{\Mnn}{\mathcal{M}_{\mathsf{AT}}}
\newcommand{\Mwe}{\mathcal{M}_{\mathsf{WE}}}
\newcommand{\blocksigma}{\widehat{\sigma}}

\newcommand{\Tmix}{T_{\text{mix}}}
\newcommand{\eps}{\varepsilon}

\newcommand{\loc}[0]{\text{-}\mathsf{Loc}}
\newcommand{\vell}[0]{\vec{\ell}}
\newcommand{\vloc}[0]{\vell\loc}
\newcommand{\tv}{{\mathrm{TV}}}

\newtheorem{theorem}{Theorem}[section]
\newtheorem{lemma}[theorem]{Lemma}

\newtheorem{proposition}[theorem]{Proposition}
\newtheorem{claim}[theorem]{Claim}
\newtheorem{corollary}[theorem]{Corollary}
\newtheorem{remark}[theorem]{Remark}
\newtheorem{definition}[theorem]{Definition}
\newtheorem{observation}[theorem]{Observation}

\title{Mixing of general biased adjacent transposition chains}
\author{Reza Gheissari\thanks{Department of Mathematics, Northwestern University. Email: \url{gheissari@northwestern.edu}} , Holden Lee\thanks{Department of Applied Mathematics and Statistics, Johns Hopkins University. Email: \url{hlee283@jhu.edu}} , and Eric Vigoda\thanks{Department of Computer Science, University of California, Santa Barbara. Email: \url{vigoda@ucsb.edu}}}
\date{}

\begin{document}
\maketitle

\begin{abstract}
    We analyze the general biased adjacent transposition shuffle process, which is a well-studied Markov chain on the symmetric group $S_n$.  In each step, an adjacent pair of elements $i$ and $j$ are chosen, and then $i$ is placed ahead of $j$ with probability $p_{ij}$.  This Markov chain arises in the study of self-organizing lists in theoretical computer science, and has close connections to exclusion processes from statistical physics and probability theory.    
    Fill (2003) conjectured that for general $p_{ij}$ satisfying $p_{ij} \ge 1/2$ for all $i<j$ and a simple monotonicity condition, the mixing time is polynomial. 
    We prove that for any fixed $\varepsilon>0$, as long as $p_{ij} >1/2+\varepsilon$ for all $i<j$, the mixing time is $\Theta(n^2)$ and exhibits pre-cutoff.
    Our key technical result is a form of spatial mixing for the general biased transposition chain after a suitable burn-in period. In order to use this for a mixing time bound, we adapt multiscale arguments for mixing times from the setting of spin systems to the symmetric group. 
\end{abstract}

\section{Introduction}

We study a natural Markov chain on biased permutations which is motivated by the study of self-organizing data structures, and is a natural generalization of fundamental stochastic processes in probability theory and statistical mechanics.   The Markov chain is defined on the set of $n!$ permutations weighted by given pairwise probabilities indicating the preference for the pairwise ordering.

For integer $n\geq 1$, let $[n]=\{1,\dots,n\}$, and let $S_n$ denote the symmetric group; thus, $S_n$ denotes the set of $n!$ permutations of the $n$ particles $\{1,\dots,n\}$. 
For a permutation $\si\in S_n$ and an index $i\in [n]$, let $\si(i)$ denote the particle in position $i$ in $\sigma$, and let $\sigma^{-1}(i)$ denote the position in $\sigma$ of the $i$-th particle.
We denote an adjacent transposition as follows: for a permutation $\si\in S_n$ and an index $i\in [n-1]$, let $\si\circ(i,i+1)$ denote the permutation obtained by swapping the particles at positions $i$ and $i+1$.

We consider the adjacent transposition Markov chain, denoted $\Mnn$, on $S_n$ weighted by given parameters $p_{i,j}$ for all $i<j$, where $0\leq p_{i,j}\leq 1$.  Given the input probabilities $p_{i,j}$ for all $i<j$, then for all $j>i$ set $p_{j,i} = 1-p_{i,j}$, and let $\vec{p}=(p_{i,j})_{1\leq i,j\leq n}$ denote the collection of probabilities.
From a state $\si_t\in S_n$, the transitions $\si_t\rightarrow\si_{t+1}$ of the Markov chain $\Mnn$ are the following: 
\begin{itemize}
\item Choose an index $i\in [n-1]$ uniformly at random.
    \item Let 
    \[  \sigma_{t+1} =
    \begin{cases} \si_t \circ (i,i+1)  & \mbox{ with probability } p_{\si_t(i+1),\si_t(i)}
    \\
\si_t & \mbox{ with probability } p_{\si_t(i),\si_t(i+1)}
\end{cases}\,.
\]
\end{itemize}

The Markov chain $\Mnn$ 
is time-reversible with unique stationary distribution $\mu$ given by
\begin{align}\label{eq:stationary-distribution}
\mu(\si) = \mu_{\vec{p},[n]}(\si)= \rc Z\prod_{1\le i<j\le n} p_{\si(i),\si(j)},
\end{align}
where the partition function is defined as 
$Z = \sum_{\si\in S_n} \prod_{i<j} p_{\si(i),\si(j)}$. 
It is an open problem to give an efficient algorithm under general conditions on $(p_{i,j})_{1\le i,j\le n}$ for approximate sampling from this distribution, or by the standard reduction between counting and sampling, for approximation of the partition function. 
One case where this is known is when $p_{i,j} = \fc{w_i}{w_i+w_j}$ with $w_1\le w_2\le \cdots \le w_n$. In this case, the partition function is a permanent of a matrix, which is efficiently approximable \cite{jerrum1989approximating,jerrum2004polynomial}. However, this relies on specific structure on $p_{i,j}$. One motivation for studying $\Mnn$ is that it provides a generic algorithm that does not rely on specific structure on $p_{i,j}$.

The efficiency of using a Markov chain for approximate sampling is quantified by its mixing time. 
To be precise, for an ergodic Markov chain $(X_t)$ on state space~$\Omega$ with transition matrix $P$ and unique stationary distribution $\mu$, the {\em mixing time} (to {$\delta\in (0,1)$}) is
\begin{align}\label{eq:tmix}
\Tmix(\delta) = \min\Big\{t\geq 0: \max_{X_0\in\Omega} \|P^t(X_0,\cdot)-\mu\|_{\tv}\leq \delta\Big\},
\end{align}
where $\|\mu - \nu \|_{\tv} = \frac{1}{2}\|\mu - \nu\|_1 $ is the total-variation distance between probability measures. 
When there is no $\delta$ argument, we use the convention $\Tmix=  \Tmix(1/4)$. 

A well-studied special case is when $p_{ij}=q$ for all $i<j$, for some constant $q$;  {in this case,~\eqref{eq:stationary-distribution} is the Mallows distribution over permutations on $[n]$~\cite{Mallows1957NONNULLRM} (see also~\cite{GnedinOlshanski,PitmanTang})}. The chain exhibits different behavior for $q=1/2$ and $q\ne 1/2$. 
When $q=1/2$, then the stationary distribution $\mu$ is uniform over all $n!$ permutations and the Markov chain $\Mnn$ is the unbiased adjacent transposition chain, which is well-studied in probability theory; see~\cite{Aldous,DiaconisSaloffCoste}. In this setting, Wilson~\cite{Wilson} proved that the mixing time to $\delta$ is $\Theta(n^3\log{n})$, and Lacoin~\cite{LacoinCutoffSSEP} established the cutoff phenomenon, identifying the constant in front of $n^3 \log n$ and showing that it is independent of $\delta$. 
The constant bias case, when $q\in (\frac{1}{2},1)$, is closely related to the asymmetric simple exclusion process (ASEP) from interacting particle systems (see the original paper of~\cite{SPITZER1970} and the book of Liggett~\cite{liggett1985interacting}), with deep connections to statistical mechanics and the KPZ universality class (see e.g.,~\cite{krapivsky2010kinetic} and~\cite{corwin2016kardar}). In this setting, mixing is faster: Benjamini, Berger, Hoffman, and Mossel~\cite{benjamini2005mixing} proved that the mixing time is $\Theta(n^2)$, and Labbe and Lacoin~\cite{LabbeLacoinCutoffASEP} established cutoff.

An additional special case of particular interest corresponds to 
the problem of linear extensions of a partial order.  In this setting, some subset of pairs have $p_{ij}=1$ (and thus require $\sigma^{-1}(i)<\sigma^{-1}(j)$) and all other pairs have $p_{ij}=1/2$ (and hence are unbiased). For this problem, Bubley and Dyer~\cite{BubleyDyer} proved that the mixing time is $O(n^3)$. 

The general version of the Markov chain $\Mnn$ has been studied in the context of self-organizing lists, where the frequency that elements are accessed is according to an underlying distribution, and elements are moved towards the front of the list when accessed to minimize the expected access time.  In that context, the chain $\Mnn$ corresponds to the Move-Ahead-One (MA1) list update algorithm, and is important in the probabilistic analysis of its stationary behavior; we refer the interested reader to Hendricks~\cite{Hendricks} and Hester and Hirschberg~\cite{HesterHirschberg} for an introduction to this area, and to Fill~\cite{Fill} for further connections.

With no conditions on $\vec{p}$, it is easy to construct slow-mixing examples, and therefore the following two conditions on the bias parameters $\vec{p}=(p_{i,j})_{1\le i<j\le n}$ have become important for  {analyzing} the general setting:
\begin{enumerate}
        \item $\ep$-positive bias: For all $1\le i<j\le n$,     $\fc{p_{i,j}}{p_{j,i}} \ge 1+\ep$.
    \item Monotonicity: For all $i<j$, $p_{i,j}\le p_{i,j+1}$ and for all $i+1<j$, $p_{i,j}\ge p_{i+1,j}$. 
\end{enumerate}

A well-known and long-standing conjecture of Fill~\cite{Fill}
is that the Markov chain $\Mnn$ has polynomial mixing time whenever $\vec p$ satisfies $0$-positive bias and monotonicity. 
Monotonicity is necessary: Bhakta, Miracle, Randall and Streib \cite{BMRS} provided an example where the mixing time is exponential in $n$ and where $\vec p$ satisfies the $0$-positive bias condition but not the monotonicity condition. 

In the positive direction, progress towards obtaining polynomial mixing time bounds for positively biased chains have relied on various special conditions on $\vec{p}$
(either assuming some algebraic relations, or some block structure across which the $p_{ij}$'s are constant).  Bhakta et al.~\cite{BMRS} proved polynomial mixing time in two restricted settings: in the first case, for $i<j$, $p_{i,j}$ only depends on $i$; in the second setting, the probabilities $p_{i,j}$ are defined based on a binary tree.  

Haddadan and Winkler~\cite{HaddadanWinkler} studied a simpler chain known as the gladiator chain where each particle has a strength $s(i)$ and the transposition of particles $i$ and $j$ is a fixed function of $s(i)$ and $s(j)$, and proved polynomial mixing when there are $k=3$ different available particle strengths.
Miracle and Streib~\cite{MiracleStreib} considered the setting where the $n$ particles are partitioned into $k$ classes $\cC_1,\dots,\cC_k$ and the particles within the same class are indistinguishable; they proved mixing time of $O(n^{2k+6})$.
Finally, Miracle, Streib, and Streib~\cite{MSS} also considered the setting of $k$ classes, and proved that if the inter-class biases satisfy $\eps$-positive bias and all classes have size $|\cC_i|\geq C(\eps)\log{n}$ for some constant $C(\eps)$, then the mixing time of $\Mnn$ is $O(n^{9})$ (independent of $k$).

We prove the following general result: for every $\eps>0$ and any $\vec{p}$ which satisfies $\eps$-positive bias, the process $\Mnn$ has mixing time of $\Theta(n^2)$, where the constant in the upper bound depends on $\eps$.
We do not require monotonicity, and note that our result does not contradict the slow mixing example of \cite[Section 6]{BMRS-arxiv} without monotonicity as they have $\eps\approx 1/n$.

\begin{theorem}\label{thm:main}
    For every $\ep>0$ and any fixed $\delta>0$, there exists $C(\ep)$ such that for any $(p_{i,j})_{1\le i< j\le n}$ which satisfies $\ep$-positive bias, the adjacent transposition Markov chain $\Mnn$ has mixing time\footnote{Here we are using $\widetilde O$ to mean up to polylogarithmic factors in $n$; the implicit constant here may, and in fact does, depend on $\eps$. However, the constant in the $\widetilde O$ is independent of $\delta$.} 
    \[ n^2(1-o(1)) \le \Tmix (\delta) \le C(\eps) n^2 + \widetilde O(n \log (1/\delta)).
    \] 
    In particular, the family has \emph{pre-cutoff} (see~\cite[Chapter 18]{LP}): there exists $C(\eps)$ such that for any sequence of $\Mnn^{(n)}$ with $\eps$-positively biased $\vec{p}\,^{(n)}$, for any fixed $\delta \in (0,\frac{1}{2})$, 
    \begin{align*}
         \limsup_{n\to\infty}  \frac{\Tmix(\delta)}{\Tmix(1-\delta) } \le C(\eps) <\infty \,.
    \end{align*}
\end{theorem}
Note this implies that the spectral gap of the chain is at least  $\Omega(C(\ep)^{-1}n^{-2})$, due to a standard inequality between mixing and relaxation time \cite[Theorem 12.5]{LP}.
It is natural to wonder whether the limit of $\Tmix/n^2$ can be pinned down, thereby showing cutoff. However, among $\eps$-positively biased $\vec{p}$, the mixing time can truly differ by order $n^2$. 
In particular, the totally asymmetric case where $p_{i,j}=1$ for all $i<j$ saturates the lower bound of Theorem~\ref{thm:main}, and the constant-bias case $p_{i,j}=q$ with $q/(1-q)=1+\eps$ for all $i<j$ saturates the upper bound with our constant $C(\eps)$. 
Therefore, any refinement into a cutoff result would require assuming a further consistency relation between the sequences $\vec{p}^{(n)}$. 

\begin{remark}\label{rem:dependence-on-epsilon}
For fixed $\eps>0$ and $n$ sufficiently large (in a way that depends on $\eps$), the constant $C(\eps)$ in Theorem~\ref{thm:main} behaves like $C(\eps) = O(1/\eps)$. On the other hand, for all $\eps,n$ we give an upper bound of $e^{\tilde O_\eps(1/\eps^2)}\text{poly}(n)$ where $\tilde O_\eps$ hides poly-logarithmic factors in $1/\eps$ (see Proposition~\ref{prop:polynomial-mixing-time}). There is a lower bound of $e^{\Omega(1/\eps)}$ that can be constructed as follows: for $n_0 = \Theta(1/\eps)$, use the slow mixing example of~\cite[Section 6]{BMRS-arxiv} on $1\le i,j\le n_0$ and for all $i,j$ that are not both in $[n_0]$, set $p_{ij} = \mathbf{1}_{i<j}$. Closing this gap would be interesting. We also note that with the added assumption of monotonicity, one would expect a $\text{poly}(1/\eps)$ dependence on $\eps$ to hold for all $\eps,n$, consistent with the conjecture of~\cite{Fill}. 
\end{remark}

\subsection{On the recent resolution of the conjecture of Fill~\cite{Fill}.}
 {Several months after the posting of this paper, Greaves and Zhu~\cite{GreavesZhu} proved Fill's conjecture that the spectral gap of the chain $\mathcal M_{\mathsf{AT}}$ is $\Omega(1/n^3)$ under the $0$-positive bias and monotonicity conditions, and $\Omega(1/n)$ under the stronger condition of $\eps$-positive bias. 
Jain and Mizgerd~\cite{JainMizgerd} as well as~\cite{GreavesZhu} then showed that the minimal spectral gap is obtained only when for some $i$, $p_{ij} =1/2$ for all $j$. These papers used purely spectral arguments  disjoint from ours.} 

 {We note that our Theorem~\ref{thm:main} is not directly comparable to these results as we do not assume monotonicity on the swap rates $(p_{ij})_{i<j}$, and we prove mixing time bounds. Without monotonicity, even with $\eps$-positive bias, their results degenerate to give no lower bound on the spectral gap. On the other hand, with monotonicity, their results allow for $0$-positive bias. It is an interesting question what more general conditions on $(p_{ij})_{i<j}$ ensure rapid mixing.
}

\section{Proof overview}
In many of the previous results where polynomial bounds are known under special settings on the bias parameters $\vec{p}$, there is some underlying algebraic structure (e.g., bijections to Markov chains on up-down walks or Dyck paths) that is used as an initial ingredient into the mixing time analysis. As a notable example, in the homogeneous case of $p_{i,j}=q$ for all $i<j$, the process which tracks the locations (but not the labels) of the first $k$ particles is Markovian,  {has a height function interpretation}, and is exactly a $k$-particle ASEP. 
With general biased $\vec{p}$, we lose these exact connections and we instead import tools from spin systems to our analysis. In particular, our overall strategy is to do a multi-scale reduction of the mixing time for $\Mnn$ on $[n]$ to that on $[\alpha n]$ for $\alpha <1$ and recurse. 

The typical way such reductions are achieved in spin system analysis is through decay of correlations between the configuration $\sigma$ on far apart regions on $[n]$. In our setting, decay of correlations does \emph{not} hold for general pinnings of the configuration, but does hold for a family of typical ones which we call $\ell$-localized: roughly speaking, this means that for all $i$, particle~$i$ is within distance $\ell$ of location $i$. 

Our strategy is captured by the following three stages: 
\begin{enumerate}
    \item Burn-in: After $O(n^2)$ steps, regardless of the initialization, the configuration is $O(\log n)$-localized, and stays so for polynomially many steps. 
    \item Polynomial bound by spatial recursion: We show that within $O(\log n)$-localized configurations, there is near-exponential decay of correlations; this property is then used together with a multi-scale block dynamics recursion to obtain $n^{O(1)}$ mixing time. 
    \item Boosting to quadratic mixing time: We use a second block dynamics reduction within the set of $O(\log n)$-localized configurations to boost the polynomial mixing time to obtain mixing in just $O(n(\log{n})^{O(1)})$ steps beyond the initial $O(n^2)$ length burn-in. 
\end{enumerate}
We describe each of these steps in more detail in the remainder of this section. In our proof overview, we will state lemmas in the form needed for establishing the polynomial upper bound. As we will describe when outlining the boosting argument, all of these steps are needed conditional on a more general form of $\ell$-localization for the boosting to optimal quadratic mixing time. Therefore, throughout the proof sketch, we will accompany lemmas with pointers to later in the paper for their more general, conditional versions.

 {In this entire section, we assume that $\vec p$ satisfies $\ep$-positive bias.}

\subsection{Burn-in to localized configurations}

Since we assume positive bias, the most likely configuration, known as the ground state, is the permutation with the particles in order $1,\dots,n$. 
We define the following set of configurations, which we refer to as $\ell$-localized, abbreviated $\ell\loc$, in which no particle deviates by more than $\ell$ positions from its ground state location.  

\begin{definition}\label{defn:localized}
For an integer $\ell\geq 0$, we say a permutation $\sigma\in S_n$ is $\ell$-\emph{localized} if the following holds:
\begin{align*}
    \si\in \ell\loc := \bigcap_{1\le k\le n} \{|\sigma^{-1}(k) - k| \le \ell\}\,.
\end{align*}
\end{definition}

We will show that in the stationary distribution $\mu$, the property $\ell_0\loc$ holds with high probability (at least $1-1/\poly(n)$) for $\ell_0 = C_0 \log n$. 

\begin{restatable}{lemma}{stationaryhighprob}
\label{lem:stationary-highprob}
There exists a constant $C_0= C_0(\eps) = O(1/\eps)$ such that the following holds:
\[
\mu((C_0\log{n})\loc)\geq 1-n^{-10}.
\]
\end{restatable}

\Cref{lem:stationary-highprob} is proved in \Cref{sec:spatial-mixing}.  The key step in the proof of \Cref{lem:stationary-highprob} is to show that for every $1\le k \le n$, under the stationary distribution, the distance of particle $k$ from location $k$ is stochastically dominated by an appropriate geometric random variable.
This follows from first showing that the identity of the particle in location 1 is stochastically dominated by a geometric random variable (Lemma~\ref{l:stoch-dom-geo}) and then applying this fact repeatedly after conditioning. 
Lemma~\ref{l:stoch-dom-geo} is shown by a  {standard} many-to-one map between configurations that have none of the first $k$ particles in location 1, to those that do, noting that moving all these particles $r$ spaces to the left increases the probability weight by a factor of at least $(1+\eps)^{kr}$. 
 {Similar arguments in the Mallows measure can be found e.g., in~\cite[Section 5]{LevinPeres16}.}

While the above results show useful properties of the stationary distribution, an additional challenge is showing that the dynamics quickly attains (and maintains) the typical stationary properties; this is often referred to as the burn-in phase of the dynamics. 
We bound the burn-in time in the following result. 

\begin{lemma}
\label{lem:single-site-burn-in-simpler}
For every $\eps$, there exists $C_0(\eps)>0$ such that for $\ell_0= C_0(\eps)\log{n}$ and any initial $\sigma_0\in S_n$, for any $t\ge C_0(\eps)\cdot n^2$,
        \[  \mathbb P({\sigma_t\in \ell_0\loc})\geq 1 - n^{-10}.
        \]
\end{lemma}

A more general version of \Cref{lem:single-site-burn-in-simpler} is stated in \Cref{lem:single-site-burn-in}, which is proved in \Cref{sec:burn-in}.
We also obtain a burn-in estimate of the form of Lemma~\ref{lem:single-site-burn-in-simpler} for a block dynamics variant of $\Mnn$, which we define in the following subsection; see Lemma~\ref{lem:restricted-block-dynamics-burn-in}.  

To prove \Cref{lem:single-site-burn-in-simpler}, we use a stochastic domination relation between ($k$-particle projections of the) $\Mnn$ and the asymmetric exclusion process (ASEP).  In ASEP, there are $k$ indistinguishable particles which lie at $k$ distinct locations on $[n]$, with a given bias parameter $q\in (1/2,1)$.  In each step, the ASEP chooses a random location $i\in [n-1]$, and if exactly one of the positions $i$ or $i+1$ contains a particle, then with probability $q$ the particle is placed ahead at position $i$, and with probability $1-q$, it is placed behind in position $i+1$; thus, there is a bias $q>1/2$ for moving the particles to the left. For the $\Mnn$ chain with general $\vec{p}$, the  process $(\mathbf 1\{\sigma_t(i)\le k\})_i$ is not Markov, but it is stochastically ``to the left'' of a $k$-particle ASEP with appropriate parameter $q$. These projections, together with known mixing time bounds of the ASEP, yield the bound on the burn-in phase as stated in \Cref{lem:single-site-burn-in-simpler}; see \Cref{sec:burn-in} for further details.  {Analogous stochastic domination arguments between exclusion processes in fixed and random environments have been leveraged in~\cite{schmid2019mixing,lacoin2024mixing}.}

\subsection{Spatial mixing within the localized set}

The role played by the localized set of configurations is that within that set of configurations, the distribution $\mu$ satisfies a certain near-exponential decay of correlations. To formalize this, we introduce the following set notations. Throughout the paper, when we write an interval $[\ell,m]$ for integers $\ell,m$, we mean the set $\{\ell,\ell+1,...,m\}\subseteq [n]$. 

\begin{definition}
    For an integer $r\geq 0$ and integers $i,j\ge 1$, let $A=A_{i,j}=[i+1,n-j]$ and let $A_r$ denote the subset of $A$ which is distance at least $r$ from the boundary.  Formally, for $A=[i+1,n-j]$, define $A_r=[i+1+r,n-j-r]$. 
    Moreover, let $A^c=[n]\setminus A$ denote the complement of the region $A$.
    \end{definition}
    
    Our spatial mixing property says that conditioned on the property $\ell\loc$, if we consider two different configurations $\eta,\bar\eta$ on $A^c$ (compatible with $\ell\loc$), then the influence of $\eta,\bar\eta$ on the configuration on $A_r$  {(the effect on the conditional distribution on $A_r$)} decays (almost) exponentially in $r$, the distance from $A_r$ to~$A^c$. In the following, we use the notation $\si(S)$ to denote the restriction of $\si$ to $S$; i.e., we consider the image as an ordered set.

\begin{restatable}{lemma}{spatialmixing}
\label{lem:spatial-mixing-simpler}
There exists a constant $C'= C'(\eps)>0$ such that  
for any integers $\ell\ge 1$, $i,j\ge 1$, and  {$0\leq r\leq (n-j-i-\ell-1)/2$}, the following holds for any pair $\eta,\bar\eta\in \ell\loc$:
\[
\|
\mu\pa{\si({A_r})\in \cdot \mid \si({A^c}) = \eta,\si\in\ell\loc} - 
\mu({\bar\si({A_r})\in  \cdot \mid   \bar\si({A^c})  =   \bar\eta,  \bar\si\in\ell\loc})
\|_{\tv} \le 
e^{- r/C'\ell}.
\]
\end{restatable}

This property is the analogue of the well-studied notion of spatial mixing from analysis of spin system dynamics~\cite{Martinelli}, where usually the spatial mixing holds for \emph{all} configurations on $A^c$. In our setting, however, that stronger statement can be seen to obviously fail: take $A = \{1,\ldots,n/2\}$ and suppose that $\eta$ assigns particles $\{1,\ldots,n/2\}$ to $A^c$ while $\bar\eta$ assigns particles $\{n/2+1,\ldots,n\}$ to $A^c$; then the available alphabets on $A$ will be completely disjoint, and there is no opportunity for $\sigma$ and $\bar\si$ to agree anywhere in $A$. 

The proof of \Cref{lem:spatial-mixing-simpler} goes by directly constructing a coupling $\sigma,\bar\si$ to agree on $A_r$ with high probability. We will prove a generalization of \Cref{lem:spatial-mixing-simpler} stated in \Cref{lem:spatial-mixing} in \Cref{sub:spatial-mixing-localized}. The coupling relies on finding what we call disconnecting positions in $A\setminus A_r$. 

\begin{definition}\label{def:disconnecting}
We say a position $k$ is \emph{disconnecting} if $\si([k])=[k]$ as \emph{unordered} sets. In words, the first $k$ particles are contained (in some order) in the first $k$ positions of $\si$. 
\end{definition}

We prove that a disconnecting point is likely for $n$ sufficiently large (as a function of $\eps$). {Such points were considered in~\cite{PitmanTang} in the Mallows case $p_{ij
}= q$ for all $i<j$, where they were called \emph{splitting points}, and closely related results were proved on their frequency.}

\begin{lemma}\label{lem:prob-of-good-point-simpler}
For every $\eps$, there exists $c(\eps)>0$ and $C(\eps)>0$ such that the following holds for all $n\ge C(\eps)$, all $\eta \in \ell\loc$, and all $k\in A_\ell = [i+1+\ell, n-j-\ell]$: 
\begin{equation}
    \mu(k\text{ disconnecting for }\si \mid \si(A^c)= \eta, \si \in \ell\loc) \ge c(\eps)\,.
\end{equation}
\end{lemma}

To prove \Cref{lem:prob-of-good-point-simpler}, we utilize the stochastic domination properties established in the proof of \Cref{lem:stationary-highprob}.  A more general version of \Cref{lem:prob-of-good-point-simpler} is stated in \Cref{lem:prob-of-good-point} and proved in \Cref{sub:spatial-mixing-localized}.
 {We note that the statement without conditioning on $\ell\loc$ follows from the corresponding result for the Mallows measure with a projection and stochastic domination argument. As we need the conditioning, we provide the full proof.}

Now note that if two configurations $\sigma([k]),\bar \sigma([k])$  are such that $k$ is a disconnecting point for both (though the $[k]$ particles may be in different orders in them), then the distribution induced on $[k]^c$ by each of them is the same, and hence the identity coupling couples them to agree on $[k]^c$. 
In every segment of length $\ell$ in $A\setminus A_r$, by Lemma~\ref{lem:prob-of-good-point-simpler}, there is a new attempt at having a disconnecting point in each of $\sigma$ and $\bar \sigma$, with  constant probability of success. Therefore, one can use a revealing procedure to prove the spatial mixing property in \Cref{lem:spatial-mixing-simpler}.

\subsection{Recursive argument to prove polynomial mixing}

We utilize the above spatial mixing and burn-in results to establish polynomial mixing time.

\begin{theorem}\label{thm:poly}
    There exists $C>0$ such that for all $\ep>0$, there exists $C'(\ep)$, such that the mixing time of $\Mnn$ satisfies
    \[ \Tmix \leq C'(\ep)\cdot n^C\,.
    \]
\end{theorem}

A more refined version of \Cref{thm:poly} is stated in \Cref{prop:polynomial-mixing-time}, which includes the quantitative dependence on $\eps$ in the mixing time.

To obtain Theorem~\ref{thm:poly}, we utilize a recursive argument.  As is common in the analysis of Glauber dynamics for spin systems like the Ising model on lattice graphs, the recursive argument we will use is comparison with a suitable block dynamics.  {An alternative approach to multi-scale arguments that has been used in the context of card shuffles and exclusion processes, including in non-homogenous environments,~\cite{LacoinCutoffSSEP,schmid2019mixing,lacoin2024mixing} is the censoring inequalities of~\cite{PeresWinkler}; however, because we do not assume the bias parameters are monotonic, the chain $\mathcal{M}_{\mathsf{AT}}$ does not itself exhibit the required monotonicity.}

We consider the following block dynamics which we denote as $\Mwe$. Define the blocks
\begin{align*}
    B_W = \{1,...,\lceil 2n/3\rceil\} \qquad B_E= \{\lfloor n/3\rfloor,...,n\}\,. 
\end{align*}
The transitions $\blocksigma_t\rightarrow\blocksigma_{t+1}$  operate as follows.  From $\blocksigma_t$, randomly select $B\in\{B_W,B_E\}$, set $\blocksigma_{t+1}(B^c)=\blocksigma_{t}(B^c)$ and sample $\blocksigma_{t+1}(B)$ from $\mu$ conditional on $\blocksigma_{t+1}(B^c) = \blocksigma_t(B^c)$.  

We utilize the following decomposition result, which is standard in the setting of spin systems.  
For a reversible, ergodic Markov chain $\M$ on state space $\Omega$ with transition matrix~$P$, 
 let  $\lambda_1=1> \lambda_2\geq\lambda_3\geq \cdots \geq \lambda_{|\Omega|}>-1$ denote the eigenvalues of $P$.  
 Let $\gamma(\M)=1-\lambda_2$ denote the {\em spectral gap} of $\M$. We state here the following simple form of the decomposition result; see \Cref{prop:block-dynamics} for a more general statement  {from which this equation follows:}
\begin{equation}
    \label{eqn:simple-block-dynamics}
        \gamma(\Mnn)\ge 
        \frac{1}{2} \cdot \gamma(\Mwe) \cdot \min_{B\in\{B_W,B_E\}}\min_{\eta} \gamma(\Mnn(B^\eta))\,,
    \end{equation}
    where $\eta$ is a fixed configuration on $B^c$ and $\Mnn(B^\eta)$ is the chain $\Mnn$ operating on $B$ (so $n$ is rescaled to $\lceil 2n/3\rceil$) conditional on the fixed configuration $\eta$ on $B^c$.

The high-level idea is to lower bound $\gamma(\Mwe)$, and then treat $\gamma(\Mnn(B^\eta))$ as (up to relabeling of the particles) an adjacent transposition shuffle on a smaller scale $[\ce{2n/3}]$ instead of $[n]$, with its own $\eps$-positively biased vector $\vec{p}$.

\begin{lemma}\label{lem:block-mixing-time-simpler}
    There exists a constant $C_{\mathsf{block}}>0$ (independent of $\eps$) 
        such that
    \[ \gamma(\Mwe)^{-1} \leq C_{\mathsf{block}}\,.
    \]
\end{lemma}

A more general version of \Cref{lem:block-mixing-time-simpler} is stated in \Cref{lem:block-mixing-time}, which is proved in \Cref{sec:poly-mixing}.
The proof goes by showing that block dynamics first burns in after $O(1)$ steps per (a variant of) Lemma~\ref{lem:single-site-burn-in-simpler}, then once it is burnt-in, the spatial mixing estimates of Lemma~\ref{lem:spatial-mixing-simpler} are applicable. Namely, if a block update occurs at time $t$ on $B_E$ and two block dynamics chains we are trying to couple, $\sigma_t, \bar \sigma_t$, are both burnt-in, then conditional on their values on $B_W\setminus B_E$, 
their marginals on $B_E\setminus B_W$ are coupled to agree with high probability. If the next update occurs on $B_W$ then the identity coupling will get them to agree everywhere. 

With the main ingredients in place  {(\eqref{eqn:simple-block-dynamics} and Lemma \ref{lem:block-mixing-time-simpler})}, we provide the proof of Theorem~\ref{thm:poly}.

\begin{proof}[\textbf{\emph{Proof of \Cref{thm:poly}}}]
Let $C'=C'(\eps)$ and $C$ be sufficiently large constants, where $C'$ depends on~$\eps$ and $C$ is independent of $\eps$. We will prove inductively  {on $n$, the following statement:} %that for all $n$,  
the maximum (over all $\eps$-positively biased $\vec{p}$) inverse spectral gap of $\Mnn$ on $[n]$ is bounded by $C' n^{C}$. 

Let $n_0$ be a sufficiently large constant as a function of $\eps$. Then the inverse gap on a segment of length less than or equal to $n_0$ is at most some constant $C'(\eps)$. (This can be seen to be bounded by a constant that is uniform over $\eps$-positively biased $\vec{p}$ by using the bound on the coupling time that comes from, in every step, making the choice that puts the adjacent particles in their preferred order.) This provides the base case.

Now assume the inductive statement holds for all $N <n$, and we will prove it holds for $n$. By applying~\eqref{eqn:simple-block-dynamics} and  {Lemma \ref{lem:block-mixing-time-simpler}}, one obtains the following upper bound on the inverse spectral gap:
\begin{align*}
    2 C_{\mathsf{block}} \max_{B\in \{B_W,B_E\}}\max_\eta \gamma(\Mnn(B^\eta))^{-1}\,.
\end{align*}
For any fixed $B$ and configuration $\eta$ on $B^c$, by relabeling the remaining particles (see Lemma~\ref{l:relabeling} for details on this relabeling), $\Mnn(B^\eta)$ is equivalent to another adjacent transposition chain with $\eps$-positively biased probabilities, on a segment now of length $2n/3$. Thus by the inductive assumption, we get the following upper bound on the inverse spectral gap:
\begin{align*}
    2C_{\mathsf{block}} \cdot C'(\eps) \cdot (2/3)^{C} n^{C}\,.
\end{align*}
As long as $C$ is large enough that $2C_{\mathsf{block}}(2/3)^C <1$, this evidently satisfies the desired bound for $n$, concluding the proof.  
\end{proof}

\subsection{Boosting to quadratic mixing time}
We conclude the proof sketch section by describing how to boost a polynomial mixing time bound for $\Mnn$ to a quadratic one. 

By Lemma~\ref{lem:single-site-burn-in-simpler}, we know that for a polynomial stretch of times, $t\in [n^2, n^4]$, say, the chain $\Mnn$ is in $\ell_0\loc$ for $\ell_0 = C_0(\eps) \log n$. 
We will show that restricted to the burn-in set (i.e., rejecting any $\Mnn$ updates that take the dynamics outside $\ell_0\loc$), the mixing time is near-linear $\widetilde O(n)$, from which the upper bound of Theorem~\ref{thm:main} would follow. 

For showing that the restricted adjacent transposition chain mixes quickly, we utilize another block dynamics reduction. Namely, we construct the following blocks 
\begin{align*}
    B_0 = \bigcup_{i\ge 0} (6 i (\log n)^3, (6i+4)(\log n)^3] \qquad B_1 = \bigcup_{i\ge0} ((6i+3) (\log n)^3, (6i+7)(\log n)^3]
\end{align*}
where each block itself consists of $O(n/(\log n)^3)$ many disjoint segments of length $4(\log n)^3$. Furthermore, note that the distance between $B_0 \setminus B_1$ and $B_1\setminus B_0$ is $(\log n)^3$ as $B_0$ and $B_1$ intersect on stretches of length $(\log n)^3$ at the ends of each constituent segment. 

We apply an analogue of the bound of~\eqref{eqn:simple-block-dynamics} except now with log-Sobolev constants to this block dynamics. The idea of the boosting is that restricted to $\ell_0\loc$, the inverse log-Sobolev constant of the block dynamics is $O(1)$ because the configuration on $B_0 \setminus B_1$ can be decoupled from the marginal on $B_1 \setminus B_0$ using a ratio form of the spatial mixing estimate Lemma~\ref{lem:spatial-mixing-simpler}, found in Lemma~\ref{lem:ratio-spatial-mixing}, on each constituent $(\log n)^3$-length segment of $B_1$, and union bounding. 

At the same time, we wish to bound the inverse log-Sobolev constant of the adjacent transposition on $B\in \{B_0,B_1\}$ by reasoning that 
\begin{itemize}
    \item restricted to $\ell_0\loc$, the individual constituent segments evolve independently, and 
    \item the mixing time on each $(\log n)^3$-length segment is bounded by $(\log n)^{O(1)}$ by application of Theorem~\ref{thm:poly} to that smaller scale. 
\end{itemize}
With these two ingredients, due to tensorization of the entropy for product chains, we would get a bound of $O(n (\log n)^{O(1)})$ for the inverse log-Sobolev constant of $\Mnn(B^\eta)$ for $B\in \{B_0,B_1\}$. 

There is a subtlety here, however, that the restriction to $\ell_0\loc$ (needed for the configurations on the constituent segments of $B_i$ to be independent) forces the block dynamics to be performed conditional on the configuration after each resampling to remain in $\ell_0\loc$. In turn, the inverse log-Sobolev bound on $B\in \{B_0,B_1\}$ needs to be for the restricted adjacent transposition shuffle that rejects any update that takes the full configuration out of $\ell_0\loc$. 
This is ultimately what forces us to actually prove all the preceding lemmas and Theorem~\ref{thm:poly} in more general forms, conditioned on certain generalizations of the $\ell\loc$ events, as alluded to at the beginning of the proof sketch.

\section{Spatial mixing for the stationary distribution}\label{sec:spatial-mixing}

In this section, we study the stationary distribution $\mu$ associated to the $\eps$-positively biased adjacent transposition shuffle. Our main aim in the section is to show that for $O(\log n)$-localized configurations, there is (near) exponential decay of correlations in the sense of Lemma~\ref{lem:spatial-mixing-simpler}. 
As hinted at in the end of the proof sketch, we actually need these stationary estimates not only under $\mu$ but also under $\mu(\cdot \mid \vec{\ell}\loc)$ where the event $\vec{\ell}\loc$ constrains each particle $i$ to be a distance $\ell_i\ge 0$ away from its ground state location $i$. Namely, the following generalizes \Cref{defn:localized}.

\begin{definition}\label{defn:generalized-localized}
        For a vector $\vec{\ell} = (\ell_i^-,\ell_i^+)_{i\in [n]}\in \R^{2n}$ with $\ell_i^\pm \ge 0$, we say  $\sigma$ is \emph{$\vec{\ell}$-localized} if 
    \begin{align*}
    \si \in \vloc :=
        \bigcap_{1\le k \le n} \{ \sigma^{-1}(k) -k \in [-\ell_k^-, \ell_k^+ ]\}.
    \end{align*}
    We let $\ell_{\max}^-=\fl{\max_k \ell_k^-}$, $\ell_{\max}^+=\fl{\max_k \ell_k^+}$, and $\ell_{\max} = \max\{\ell_{\max}^-, \ell_{\max}^+\}$.  
\end{definition}

    Note the earlier definition of $\ell\loc$ in \Cref{defn:localized} is a special case of $\vec{\ell}\loc$ with the constant vector $\vec{\ell} =(\ell,\ldots, \ell)$.
    We use the same terminology for a partial (injective) assignment $\si:A\to [n]$ where $A\subeq [n]$, where the intersection is now only over $k\in\si(A)$.  {To avoid the degenerate case, we henceforth assume that $\ell_{\max}\ge 1$.}

In order to ensure connectivity (under moves of $\Mnn$) of the set of $\vec{\ell}$-localized permutations, we restrict attention to admissible localizations, defined below.

\begin{definition}    A vector $\vec{\ell}\in \R^{2n}$ is \vocab{$n$-admissible} if 
    \begin{align*}
         j-\ell_j^-  \le k-\ell_k^- \qquad \text{ and } \qquad j + \ell_j^+  \le k + \ell_{k}^+ \qquad \forall j<k\,.
    \end{align*}
    We let $\mathcal L_n$ denote the set of $n$-admissible vectors. 
\end{definition}

The following is the general form of our spatial mixing lemma, where the spatial mixing given $(C_0 \log n)\loc$ holds even if you condition further on an admissible $\vec{\ell}\loc$. Establishing this estimate will be the main purpose of this section.  The following result is a generalization of \Cref{lem:spatial-mixing-simpler}.

\begin{lemma}
\label{lem:spatial-mixing}
There exists a constant $C'(\eps) = e^{O(1/\ep)}$ such that the following holds. 
Consider two permutations $\eta,\bar\eta\in \vloc$ and let $A=[i+1,n-j]$, $A^c=[n]\bs A$, and 
$A_r = [i+1+r, n-j-r]$. Then for $r\ge \ell_{\max}$, $r\le \fc{n-j-i-\ell_{\max}-1}2$, 
\begin{multline*}
\big\|
\mu\pa{\si({A_r})\in \cdot \mid \si(A^c) = \eta, \si\in \vloc} - 
\mu\pa{\bar\si({A_r})\in \cdot \mid \bar\si(A^c)  = \bar\eta,   \bar\si\in \vloc}
 \big\|_{\tv} \\
\le 
\exp(-r/C'(\eps)\ell_{\max}).
\end{multline*}
If $i=0$, the above holds with $A_r=[1,n-j-r]$, and if $j=0$, the above holds with $A_r=[i+1 + r,n]$.
\end{lemma}

Note that after a burn-in to $\ell_0\loc$ for $\ell_0 = C_0 \log n$, this implies exponential decay of correlations of the configuration beyond distances $O(\log n)$.

\subsection{Self-reducibility of the set of localized configurations}

Before working towards the proof of \Cref{lem:stationary-highprob}, we prove some basic properties about the connectedness and self-reducibility of the Markov chain when restricted to $\vec{\ell}\loc$ for $\vec{\ell}$ admissible.

\begin{claim}
\label{cl:ooo}
    If $\sigma \in \vloc$, and $\sigma'$ swaps %(any, not necessarily adjacent) 
    two  {adjacent} particles that are out of order in $\sigma$, then $\sigma' \in \vloc$. In particular, the adjacent transposition chain restricted to an $\vell\loc$ for $\vell$ admissible is still ergodic. 
\end{claim}

\begin{proof}
Suppose that configuration $\sigma$ has particle $i$ at location $v$ and particle $j$ at location $v+1$ for $i>j$ so that they are ``out of order" and swapping the locations of particles $i$ and $j$. Suppose by way of contradiction that after the update, $v+1 - i \notin [-\ell_i^-, \ell_i^+]$ or $v-j\notin [-\ell_j^-,\ell_j^+]$, but before the update $v-i \in [-\ell_i^-, \ell_i^+]$ and $v+1-j \in [-\ell_j^-, \ell_j^+]$.

In the first case, $v+1 -i\notin [-\ell_i^-,\ell_i^+]$ while $v-i \in [-\ell_i^-,\ell_i^+]$, then necessarily $v= i+\ell_i^+$ and $v+1 > i+\ell_i^+$. But then, since $i>j$ one also has $v+1 > j+ \ell_j^+$ by admissibility of $\vell$, which gives a contradiction. 

Symmetrically, if $v-j \notin [-\ell_j^-,\ell_j^+]$ while $v+1 -j \in [-\ell_j^-,\ell_j^+]$ then necessarily $v+1=j-\ell_j^-$ so $v< j- \ell_j^-$ but then $i>j$ and admissibility of $\vell$ implies $v<i-\ell_i^-$ which is a contradiction. 

    The ergodicity follows from the fact that from any configuration, there is a path of adjacent transpositions that are swapping particles that are out of order, that takes it to the fully ordered configuration $(1,2,...,n)$. 
\end{proof}

We will need the fact that $\mu$ conditioned on the values of $\si(1),\ldots, \si(i)$ is also a distribution with $\ep$-positive bias, after relabeling $[n]\bs \{\si(1),\ldots, \si(i)\}$. Moreover, $\vell$-locality induces some $(n-i)$-admissible $\vec{\ell}'$-locality on $[n]\setminus [i]$. This is formalized below.  {Note that it is necessary that $A$ be an interval for part 3.}
\begin{lemma}
\label{l:relabeling}
    Let $A=[i+1,n-j]$ and $A^c=[n]\setminus A = [i]\cup[n-j+1,n] $. 
    Let $\vell$ be $n$-admissible.
            For $b:A^c\to [n]$ that is $\vell$-localized, define the relabeling map $r_{b}:[n-i-j]\to [n]\bs b(A^c)$ as the unique strictly increasing (one-to-one) map between those sets.
    Given $\si:A\to [n]$ (such that $\si\opl b\in S_n$\footnote{$\si\opl b$ denotes the function that takes values $\si(x)$ when $x\in A$ and $b(x)$ when $x\in A^c$.}), define the restriction $R_{b}\si\in S_{n-i-j}$ by
    \[
R_b\si(x) = r_b^{-1}(\si(i+x)) \qquad \text{for $x\in [n-i-j]$}\,.
    \]
    Then the following hold.
    \begin{enumerate}
        \item 
        (Labels preserved (up to translation) away from boundary)
        Let $\si\in S_n$ and $b=\si(A^c)$. Then 
        $r_{b}(k) = k+i$ for $k\in [1+\ell_{\max}^-, n-i-j-\ell_{\max}^+]$.
                \item 
        (Locality maintained)
        Let $\si\in S_n$. 
        $\si\in \vloc$ if and only if letting $b=\si(A^c)$, $b$ is $\vell$-localized and  
        $R_b\si\in \vell'\loc$, where $\vec{\ell}'=R_b\vell$ is defined by 
        \begin{align*}
        \ell_{k}^{\prime-} &= 
                \min\{\ell_{r_b(k)}^- +k+i - r_b(k), k-1\}
        \\ 
        \ell_{k}^{\prime+} &=
        \min\{\ell_{r_b(k)}^+ + r_b(k)-k-i, n-i-j-k\}
        .
        \end{align*}
        Moreover, 
        $\vell'$ is $n-i-j$-admissible with $\ell_{\max}^{\prime\pm}\le \ell_{\max}^{\pm}$, respectively.
        \item (Measure on restriction)
        Let $\mu$ be the probability measure on $S_n$ given by $p_{k,k'}$, and let $\mu_{|b}$ be the measure conditioned on $\set{\si\in S_n}{\si({A^c})=b}$. 
        Then $(R_b)_* \mu_{|b}$\footnote{We use the notation $f_*\mu$ to denote the pushforward of the measure $\mu$ by the function $f$.} is the probability measure on $S_{n-i-j}$ given by~\eqref{eq:stationary-distribution} with $\vec{q}$ given by $q_{k,k'} = p_{r_b(k),r_b(k')}$. If $\vec{p}$ satisfies $\ep$-positive bias, then $\vec{q}$ satisfies $\ep$-positive bias. 
            \end{enumerate}
\end{lemma}
 {
As an example, take $n=10$, $i=j=2$, so that $A=[3,8]$ and $A^c = \{1,2,9,10\}$. Suppose $\ell_{\max}^-=\ell_{\max}^+=1$ and $b:A^c\to [10]$ is given by the following:
    \begin{center}
    \begin{tabular}{|c|c|c|c|c|}
        $k$ & 1 & 2 & 9 & 10\\
        \hline
        $b(k)$ & 1 & 3 & 8 & 10\\
    \end{tabular}
    \end{center}
The increasing relabeling map is given by the following:
\begin{center}
    \begin{tabular}{|c|c|c|c|c|c|c|}
        $k$ & 1 & 2 & 3 & 4 & 5 & 6 \\
        \hline
        $r_b(k)$ & 2 & 4 & 5 & 6 & 7 & 9 \\
    \end{tabular}
\end{center}
Note that away from the boundary, for $k\in [2,5]$, $r_b(k) = k+2$.
}
\begin{proof}
We go through the proof item by item. Note that in the definition of $\vec{\ell}\loc$, the following is equivalent: 
\begin{align*}
    \bigcap_{1\le k \le n} \{\sigma^{-1}(k) - k \in [-\ell_k^-,\ell_k^+] \} = \bigcap_{1\le k\le n} \{\sigma(k) - k \in [-\ell^+_{\sigma(k)}, \ell^-_{\sigma(k)}]\}\,.
\end{align*}
\begin{enumerate}
    \item We only need to consider when the given interval is non-empty. If $k\in [i]$, then because $b$ is $\vell$-localized, $b(k)\le k+\ell_{b(k)}^-\le i+\ell_{\max}^-$. 
    If $k\in [n-j+1,n]$, then $b(k)\ge k-\ell_{b(k)}^+\ge  n-j+1 - \ell_{\max}^+$.
    Therefore, we can write the set of particles not assigned by $b$ as 
    \begin{multline*}
[n]\bs b(A^c) = 
([i+\ell_{\max}^-]\bs b([i]))
\cup 
[i+\ell_{\max}^-+1, n-j-\ell_{\max}^+]\\
\cup 
([n-j+1-\ell_{\max}^+, n]\bs b([n-j+1,n]))\,.
    \end{multline*}
    We must have that $r_b$ maps $[\ell_{\max}^-]$, 
    $[\ell_{\max}^-+1, n-i-j-\ell_{\max}^+]$,
    $[n-i-j-\ell_{\max}^++1, n-i-j]$ to these three sets, respectively. Considering the middle set gives the conclusion.
    \item 
    Note that $\si\in \vloc$ if and only if $b$ is $\vell$-localized and for all $k\in [n-i-j]$, 
    \begin{align}
\si^{-1}(r_b(k)) \in [r_b(k)-\ell_{r_b(k)}^-, r_b(k)+\ell_{r_b(k)}^+]\,.
\label{e:loc-int1}
    \end{align}
    The condition that $R_b\si\in\vell'\loc$ is that 
    \begin{align*}
        (R_b\si)^{-1}(k) - k \in [-\ell_k^{\prime-}, \ell_k^{\prime+}]
 \end{align*}
 or equivalently that 
 \begin{align}
        \si^{-1}(r_b(k)) &\in [i+k-\ell_k^{\prime-}, i+k+\ell_k^{\prime+}].
        \label{e:loc-int2}
        \end{align}
    Note that if not for the $\min$  {(clamping)} operation, 
    $\vell'$ was chosen exactly so that \eqref{e:loc-int1} and \eqref{e:loc-int2} coincide.  {(That is, equating them would give $\ell_{k}^{\prime-}= \ell_{r_b(k)}^- +k+i - r_b(k)$ and $\ell_{k}^{\prime+} =
    \ell_{r_b(k)}^+ + r_b(k)-k-i$.)}
    For admissibility, note that these intervals are monotonically moving to the right (since $\vell$ is admissible) and $r_b$ is increasing.
    Finally, note that taking the min simply truncates the intervals to $[1,n]$, and this does not change admissibility.
    
    It remains to check that $\ell_{\max}^{'\pm}\le \ell_{\max}^{\pm}$. For that it is sufficient to show
            \begin{align*}
    \forall k&\ge \ell_{\max}^-+1\,,&
    \ell_{r_b(k)}^- + k+i-r_b(k)&\le \ell_{\max}^-\,,\\
    \forall k&\le n-i-j-\ell_{\max}^+\,,&
    \ell_{r_b(k)}^+ + r_b(k) -k-i&\le \ell_{\max}^+\,.
    \end{align*}
    In light of part 1, $r_b$ maps $[\ell_{\max}^-]$ to $[i+\ell_{\max}^-]\bs b([i])$. Because $r_b$ is defined on an interval of integers, $r_b(x)-x$ is non-decreasing. Therefore, $r_b(k)\ge k+i$ in the first case, showing the first statement. 
    Similarly, $r_b$ maps $[n-i-j-\ell_{\max}^++1,n-i-j]$ to $[n-j+1-\ell_{\max}^+, n]\bs b([n-j+1,n])$, so $r_b(k)\le k+i$ in the second case, showing the second statement.
                        \item 
    Suppose $\si'=R_b\si$. Then 
    \begin{align*}
(R_b)_* \mu_{|b}(\si')
= \mu_{|b}(\si) 
&\propto \prod_{i<x<y\le n-j} p_{\si(x),\si(y)}\\
&= \prod_{i<x<y\le n-j} p_{r_b(\si'(x-i)), r_b(\si'(y-i))} = 
\prod_{1\le x<y\le n-i-j}q_{\si'(x),\si'(y)}.
    \end{align*}
    The fact     that $\ep$-positive bias is maintained follows from $r_b$ being increasing. \qedhere
\end{enumerate}
\end{proof}

\subsection{Bounds on deviations of particle locations}

We begin by showing a geometric decay bound on the label of the particle at location~$1$. 

\begin{lemma}\label{l:stoch-dom-geo}
Suppose $\vec{p}$ is $\eps$-positively biased and $\vec{\ell}\in \mathcal L_n$. For every $k\ge 1$, 
    \begin{align*}
        \mu(\sigma(1)\notin [k] \mid \sigma \in \vec{\ell}\loc)\le \frac{1}{(1+\eps)^k}\,.
    \end{align*}
    Hence, the particle label at location $1$ is stochastically dominated by a $\text{Geom}(\frac{\eps}{1+\eps})$ random variable. 
\end{lemma}
\begin{proof}
    Consider the following function $\Phi:S_n\to S_n$ which maps any $\si$ such that $\si(1)\nin [k]$ to a $\si'$ such that $\si'(1)\in [k]$ as follows: if the first particle in $[k]$ is at position $j$, then move all particles in $[k]$ $j-1$ spaces to the left. 
    Formally, if the first $k$ particles are in positions $j=i_1<\ldots <i_k$, then 
        \[
\Phi(\si) =\si \circ  \bigcirc_{\ell=0}^{j-2}\big((i_1-\ell-1,i_1-\ell)\circ\cdots \circ (i_k-\ell-1,i_k-\ell)\big)\,,
    \]
     {where $\bigcirc_{\ell=a}^b \si_\ell$ denotes $\si_a\circ\si_{a+1}\circ\cdots \circ \si_{b}$; note }
    the last index in the composition is the right-most one.
        We claim that if $\si\in \vloc$, then $\Phi(\si)\in \vloc$. This is because each transposition switches the order of two particles which are out of order: the particle on the left is $>k$, and the particle on the right is $\le k$.

    Observe that if $\si$ and $\si(i)>\si(i+1)$, then considering the only factors which do not cancel out in $\prod_{1\le i<j\le n}p_{\si(i){,}\si(j)}$,
    \begin{align}
    \label{e:one-transpose}
\fc{\mu(\si\circ (i,i+1))}{\mu(\si)} \ge \fc{p_{\si(i+1),\si(i)}}{p_{\si(i),\si(i+1)}} \ge 1+\ep.
    \end{align}
    Now using the fact that $\Phi^{-1}(\sigma)$ for $\sigma: \sigma(1)\in [k]$ partition $\{\sigma: \sigma(1)\notin [k]\}$,
    \[
\fc
{\mu(\si(1)\nin [k] \mid  \si\in \vloc)}
{\mu(\si(1)\in [k] \mid  \si\in \vloc)}
= 
\fc{\sum_{\si(1)\in [k],\si\in \vloc}\mu(\Phi^{-1}(\si)\cap \vloc)}{\sum_{\si(1)\in [k],\si\in \vloc}\mu(\si)}
\le \max_{\si\in \vloc} \fc{\mu(\Phi^{-1}(\si))}{\mu(\si)}\,.
    \]
    Using \eqref{e:one-transpose} and the fact that in each of those transpositions, we are switching the order of two particles which are out of order, we can bound the right-hand side by 
    \begin{align*}
\fc{\mu(\Phi^{-1}(\si))}{\mu(\si)} &= 
\fc{\sum_{j\ge 1} \mu\pa{\si\circ \bigcirc_{\ell=0}^{j-2}\big((i_1-\ell-1,i_1-\ell)\circ\cdots \circ (i_k-\ell-1,i_k-\ell)\big)}}{\mu(\si)} \\
&\le \sum_{j\ge 1}\rc{(1+\ep)^{jk}}.
    \end{align*}
    Therefore (using that $\mu(A)  r \ge \mu(A^c)$ implies $\mu(A)\ge 1/(1+r)$), 
    \begin{align*}
\mu(\si(1)\nin [k] \mid  \si\in \vloc)
\le 1-  \Big(1+\sum_{j\ge 1}\rc{(1+\ep)^{jk}}\Big)^{-1} = \rc{(1+\ep)^k}\,,
    \end{align*}
    as claimed. 
\end{proof}

We now prove that configurations from $\mu$ are $C_0 \log n\loc$ with high probability. Though this proof would hold even conditioned on $\vec{\ell}\loc$ for admissible $\vell$, we only ever use it and therefore prove it unconditionally; the conditional version would follow mutatis mutandis.

For this purpose, we begin with a corollary of Lemma~\ref{l:stoch-dom-geo} which inverts that lemma to show that the position of the first particle with index less than or equal to $k$ is stochastically below a geometric random variable. 

\begin{corollary}\label{cor:particle-location-stoch-dom-geo}
    Suppose $\vec{p}$ is $\eps$-positively biased and $\vec{\ell} \in \mathcal L_n$. For $k \ge 1$, let 
    \begin{align*}
       X = \min\{s: \sigma(s) \le k\} \,.
    \end{align*}
    Then, the law of $X$ under $\mu (\cdot \mid \vell\loc)$ is stochastically below a geometric random variable with success probability $1- (1+\eps)^{-k}$.
\end{corollary}

\begin{proof}
    Consider the event that $X > s$. This is equivalent to $\{\sigma(1)\notin [k]\}\cap \cdots \cap \{\sigma([s])\notin [k]\}$. Thus, 
    \begin{align*}
        \mu(X > s \mid \vell\loc)  = \prod_{i=1}^s \mu(\sigma(i)\notin [k] \mid \sigma(1)\notin [k],....,\sigma(i-1)\notin [k],\vloc)\,.
    \end{align*}
    By increasing the conditioning to fully reveal the particle assignments at locations $1,...,i-1$, then the law on the remainder of the locations $[i,n]$ is an $\eps$-positively biased instance of the same problem, with all particles $[k]$ still available. Thus, by Lemma~\ref{l:relabeling}, each term above is equivalent to some $\mu'(\sigma(1)\notin [k] \mid \vell'\loc)$ for some $\mu'$ corresponding to a $\eps$-positively biased sequence $\vec{p}$, and an $n-i$-admissible $\vell'$. In particular, by Lemma~\ref{l:stoch-dom-geo},
    \begin{align*}
        \mu(X > s \mid \vell\loc) \le \frac{1}{(1+\eps)^{ks}}\,,
    \end{align*} 
    which implies the claimed stochastic domination. 
\end{proof}

\subsection{Typical configurations are logarithmically localized}

We can now prove \Cref{lem:stationary-highprob}.

\begin{proof}[\textbf{\emph{Proof of \Cref{lem:stationary-highprob}}}]
We use a coupon collecting argument to bound the location $r$ where $\si(r) = k$, showing that  {for arbitrarily large $C$, } $r=k+O\pa{\rc{\ep}\log \pf{n}{\ep}}$ with probability at least $1-n^{-C}$  {where the constant in the $O(\cdot)$ depends on $C$}.
Given $k$, let $X_0=0$ and $X_1, X_2,\ldots, X_k$ be such that 
\[
X_{s} = \min\set{r>X_{s-1}}{\si(r)\le k}.
\]
Note that $\si^{-1}(k)\le X_k$. 
We claim that conditional on $\sigma([X_{s-1}])$, the random variable $Y_s = X_s-X_{s-1}$ is stochastically dominated by a geometric random variable with success probability $1-\rc{(1+\ep)^{k+1-s}}$. Indeed, revealing $\sigma([X_{s-1}])$, the remaining particles to be assigned to $[X_{s-1}]^c$ have $k+1-s$ many of the letters of $[k]$. Relabeling that remainder by Lemma~\ref{l:relabeling} and applying Corollary~\ref{cor:particle-location-stoch-dom-geo}, we get the stochastic domination for $Y_s$. 

We then use the following tail bound on sums of independent geometric random variables. 

\begin{lemma}\label{lem:geometric-sum-tail}
    Let $Y_i\sim \mathsf{Geometric}(a_i)$ be independent for $i \in [k]$. Then letting $S=\sumo ik \rc{a_i}$, for some universal constant $C$, 
    \[
\Pj\Big({\sumo ik Y_i \ge S+u}\Big)
    \le     \exp\Big({-
    C\min\Big\{{
        \fc{u^2}{\sumo ik \fc{1-a_i}{a_i^2}},
        u\min_i a_i
    }\Big\}
    }\Big).
\]
\end{lemma}
\begin{proof}
The moment generating function of $Y\sim \mathsf{Geometric}(a)$ is $\E [e^{tY}] = \fc{ae^t}{1-e^{t}(1-a)}$. 
We have by Taylor expansion that for $t<\fc a2$,
\[
\E e^{tY-\fc ta} = 
e^{t(1-\rc a) - \log \pa{1-(e^t-1) \pf{1-a}{a}}} = 
e^{O\pa{t^2\pa{\rc a -1 + \pa{\rc a -1}^2}}} = 
e^{O\pa{t^2\pa{\fc{1-a}{a^2}}}}\,,
\]
where the constant in the big-$O$ is uniform in $(a_i)_i$ and $k$.
Hence,
\begin{align*}
    \Pj\Big({\sumo ik Y_i \ge S+u}\Big)
    \le \fc{\E e^{t\sumo ik Y_i}}{e^{t(S+u)}}
    \le 
    e^{O\pa{t^2\sumo ik \fc{1-a_i}{a_i^2}
        }-tu}\,.
        \end{align*}
Optimizing over $t\le \rc 2 \min a_i$ gives the lemma.
\end{proof}

We can now complete the proof of the high probability properties of the stationary distribution stated in \Cref{lem:stationary-highprob}.

In our case, for $a_i = 1-\rc{(1+\ep)^i}$ for $1\le i\le k$, 
splitting into the cases  {$i\le \rc{\ep}$ and $i>\rc{\ep}$}, we get that the contribution to $\sumo ik \fc{1-a_i}{a_i^2}$ is $O\prc{\ep^2}$ and $O\prc{\ep}$, respectively. Note
\begin{align*}
\sumo ik \rc{a_i} = 
\sumo i{\fl{1/\ep}} \rc{a_i} + 
\sum_{i=\fl{1/\ep}+1}^k \rc{a_i}
&= 
\sumo i{\fl{1/\ep}} O\pf{1}{i\ep} + 
\sum_{i=\fl{1/\ep}+1}^k \ba{1+O\pa{\prc{1+\ep}^{i-\fl{1/\ep}}}} \\
&=
O\pa{\rc{\ep} \log\prc{\ep}} + \pa{k+O\prc{\ep}}.
\end{align*}
Hence, by Lemma~\ref{lem:geometric-sum-tail}, there is $C$ so that for any $0<\de<1$,
\[
\mu\Big({
\si^{-1}(k) \ge k+\fc{C}{\ep}\log \big(\rc{\ep\de}\big)
}\Big)\le \fc\de2.
\]
For $k'=k-\fc{C}{\ep}\log \prc{\ep\de}$ for appropriate $C$, by a complementary stochastic domination argument using Lemma~\ref{l:relabeling} and Corollary~\ref{cor:particle-location-stoch-dom-geo}, we have 
\begin{align*}
    \mu(\si^{-1}(k) > k') 
    &\ge 
    \prodo i{k'} \mu(\si(i)\in [k-1] \mid \si(1)\in [k-1],\ldots, \si(i-1)\in  [k-1])\\
    &\ge \prodo i{k'} \pa{1-\rc{(1+\ep)^{k-i}}}
    \ge 1-\sumo i{k'} \rc{(1+\ep)^{k-i}}
    \ge 1-\sum_{i=k-k'}^\iy \rc{(1+\ep)^i} \\
    &\ge 1-\fc\de2.
\end{align*}
By a union bound, 
\[
\mu\pa{
|\si^{-1}(k)-k| \ge \fc{C}{\ep}\log \pf{n}{\ep\de}
}\le \de\,.
\]
Taking  {$\delta = n^{-\widetilde C}$ for arbitrarily large $\widetilde C$} concludes the proof of \Cref{lem:stationary-highprob}.
\end{proof}

\subsection{Spatial mixing for localized configurations}\label{sub:spatial-mixing-localized}

It remains to use a coupling argument to establish Lemma~\ref{lem:spatial-mixing}. As described in the proof sketch, this will rely on the notion of disconnecting points which isolate influence from the configuration on one side of the point from the other.   
Recall from Definition~\ref{def:disconnecting}, for a permutation $\si\in S_n$ and a position $k\in [n]$, we refer to position $k$ as \emph{disconnecting} if $\si([k])=[k]$ as unordered sets. 

The following lemma shows that disconnecting points are common in intervals of length at least  $n_0=  e^{O(1/\eps)}$.  A simpler form of the following lemma was stated in \Cref{lem:prob-of-good-point-simpler}.

\begin{lemma}\label{lem:prob-of-good-point}
The following hold for all $\eps$-positively biased and all $\vec{\ell}\in \mathcal L_n$.
\begin{enumerate}
    \item For any $k\in [n]$,  $\mu(k\text{ disconnecting for }\si\mid \si\in \vloc)\ge e^{-O(1/\ep)}$.
    \item 
    Let $A^c=[i]\cup [n-j+1,n]$. 
    For $k\in [i+\ell_{\max}^-    , n-j-\ell_{\max}^+]$ and any $b:A^c\to [n]$ that is the restriction of some $\si\in \vloc$, 
    \[
\mu(k\text{ disconnecting for }\si\mid \si\in \vloc, \si(A^c)=b)\ge e^{-O(1/\ep)}. 
    \]
\end{enumerate}
\end{lemma}
\begin{proof} We prove the lemma item by item. 
\begin{enumerate}
    \item The probability we are interested in can be written as 
    \begin{align*}
        \prod_{i=1}^{k}  
        \mu(\si(i)\le k\mid\si\in \vloc; \si(1)\le k, \ldots, \si(i-1)\le k) \,.
    \end{align*}
    By revealing $\sigma(1),...,\sigma(i-1)$ and applying Lemma~\ref{l:relabeling}, each of the terms in the product, Lemma~\ref{l:stoch-dom-geo} are lower bounded by $1-\frac{1}{(1+\eps)^{k+1-i}}$, so that 
    \begin{align*}
       \mu(k\text{ disconnecting for }\si\mid \si\in \vloc) \ge  \prod_{i=1}^{k} \Big({1-\rc{(1+\ep)^{k+1-i}}}\Big)
                        = e^{-O(1/\ep)}.
    \end{align*}
    \item 
    Using the notation of \Cref{l:relabeling}, by that lemma  $(R_b)_*\mu$
    is a measure on biased permutations that satisfies $\ep$-positive bias. % and monotonicity. 
    Applying item 1, if we let $\sigma' = R_b \sigma$, we obtain 
    \[
((R_b)_*\mu)(k-i\text{ disconnecting for }\si'\mid \si'\in R_b\vloc)\ge e^{-O(1/\ep)}\,.
    \]
    It remains to show that if $k-i$ is disconnecting for $\si'=R_b\si$ and $\si(A^c)=b$, then $k$ is disconnecting for $\si$. To show this claim, note that we have $\si([i])=b([i])\subeq [i+\ell_{\max}^-]\subeq [k]$ and by $k-i$ being disconnecting for $\sigma'$, 
    \[
    [k-i] \supeq (\si')^{-1}([k-i]) = \si^{-1}(r_b([k-i]))-i\,,
    \]
    so $\si^{-1}(r_b([k-i]))\subeq [k]$. 
    Note $r_b([k-i])\subeq [k]$ because $k\le n-j-\ell_{\max}^+$. 
    Finally, note that $r_b([k-i])$ and $b([i])$ are disjoint subsets of $[k]$ and the union has $k$ elements, so the union equals $[k]$, and we obtain $\si^{-1}([k])=[k]$. \qedhere
\end{enumerate}
                                \end{proof}

We can now use disconnecting points to prove the spatial mixing lemma.

\begin{proof}[\textbf{\emph{Proof of \Cref{lem:spatial-mixing}}}]
We construct a coupling of the distributions. The idea is to sample independently and look for common disconnecting points between the two configurations, revealing from the outside-inwards. On the one hand, we have an independent attempt every $\ell_{\max}$ indices, and on the other upon encountering a pair of common disconnecting points, the marginal distribution on the interval between them is identical.

Formally, we proceed as follows.
First, by rounding down $r$ and adjusting the constant as necessary, it suffices to consider $r=r'\ell_{\max}$. 
Consider independently sequentially sampling, 
\begin{align*}
    \si(i+k) \,&\mid\,\si(A^c)=\eta,\, \si({[i+1,i+k-1]}), \,\si\in \vloc\,,\\
    \bar\si(i+k) \,&\mid\,\bar\si(A^c)=\bar\eta, \,\bar\si([i+1,i+k-1]), \bar\si\in \vloc\,,
\end{align*} 
until we arrive at a $k$ such that $k=k'\ell_{\max}$, $k'\le r'$ such that $i+k$ is disconnecting for both $\si$ and $\bar\si$ (and declare failure otherwise). Note that this event is measurable with respect to the sigma-algebras generated by $\sigma([i+k]),\bar \sigma([i+k])$. Conditioned on $\si(A^c)=\eta$, $\si({i+1,i+k'\ell_{\max}})$, and $\si\in \vloc$, the probability that $i+(k'+1)\ell_{\max}$ is disconnecting is at least $\rho := e^{-O(1/\ep)}$ by item (2) of \Cref{lem:prob-of-good-point}. 
Thus, the probability that we  fail to find such a disconnecting point for both $\si$ and $\bar\si$ is $(1-\rho^2)^{ {r'}} \le e^{-e^{-O(1/\ep)}r/\ell_{\max}}$.

Supposing that $i+k$ where $k=k'\ell_{\max}$ is disconnecting for both $\si$ and $\bar\si$, we now independently sequentially sample in reverse,
\begin{align*}
\si(n-j+1-m) \,&\mid\, \si(A^c) = \eta, \,\si([i+1,i+k]\cup [n-j+2-m,n-j]), \, \si\in \vloc\\
\bar\si(n-j+1-m) \,&\mid\, \bar\si(A^c) = \bar\eta,\,  \bar\si\pa{[i+1,i+k]\cup [n-j+2-m,n-j]},\, \bar\si\in \vloc.
\end{align*}
 {Define a position $x$ be reverse-disconnecting for $\si$ if $\si([x,n]) = [x,n]$ as an unordered set.} The same proof as in \Cref{lem:prob-of-good-point} and as above shows that the probability we fail to find a reverse-disconnecting point $n-j+1-m$ where $m=m'\ell_{\max}$, $m'\le r'$ for both $\si$ and $\bar\si$ with probability $e^{-e^{-O(1/\ep)}r/\ell_{\max}}$. Note by the upper bound assumption on $r$, there is a buffer of size at least $\ell_{\max}$ as necessary for \Cref{lem:prob-of-good-point}.  {(Alternatively, apply the argument directly to $\pi\circ \si\circ \pi$ where $\pi(i) = n+1-i$, noting that for $\si\sim \mu_{\vec p, [n]}$, the distribution of $\pi\circ \si\circ \pi$ is given by the probabilities $(p_{n+1-j,n+1-i})_{1\le i, j\le n}$ and $\si$ is $(\ell^-,\ell^+)$-localized iff $\pi \circ \si \circ \pi$ is  
$(\tilde \ell^-, \tilde \ell^+)$-localized, where $\tilde \ell^-_k = \ell^+_{n+1-k}$ and $\tilde \ell^+_k = \ell^-_{n+1-k}$.)}

If we have found both a disconnecting point $i+k$ and reverse-disconnecting point $n-j+1-m$, then we have $[i+k]=\si([i+k])=\bar\si([i+k])$, and $[n-j+1-m,n] = \si([n-j+1-m,n])=\bar\si([n-j+1-m,n])$. Letting $B=[n]\bs ([i+k]\cup [n-j+1-m,n])$, 
we note $B\supeq A_r$ and  the distribution of $\si(B)$, $\bar \si (B)$ conditioned on the set values (and $\si,\bar\si\in \vloc$) is the same. 
        
Finally, note that if either $i=0$ or $j=0$, then there are no boundary conditions on that side, and it suffices to find a disconnecting point on the opposite side.  {Hence a buffer of $r$ is unnecessary on that side.}
\end{proof}

\subsection{Ratio spatial mixing}
We also prove a ratio form of the spatial mixing lemma above. In the context of spin systems with finite energy, moving from one to the other with is fairly standard, but in our context, the interactions are hard (the choice of particle at site $i$ fully changes the alphabet available to site $j$), so a modification of the argument is needed.

\begin{lemma}\label{lem:ratio-spatial-mixing}
    In the same setting as Lemma~\ref{lem:spatial-mixing}, with possibly a different choice of $C'(\eps) = e^{O(1/\eps)}$, one has for $r\ge C'(\eps)\ell_{\max}$ that 
    \begin{align*}
        \max_{\zeta} \left|\frac{\mu(\sigma(A_r)= \zeta \mid \sigma(A^c)=\eta, \sigma \in \vec{\ell}\loc)}{\mu(\bar \sigma(A_r)= \zeta \mid \bar \sigma(A^c)=\bar \eta, \bar \sigma \in \vec{\ell}\loc)} - 1\right|  \le \exp\left( - \fc{r}{C'(\eps) \ell_{\max}}\right)
    \end{align*}
    where the maximum is over all $\zeta$ for which the probabilities involved are non-zero (that set of $\zeta$ being the same so long as $r\ge 2\ell_{\max}$). 
\end{lemma}
% The main idea of the proof is to couple the distributions under conditioning by $\si(A^c) = \eta$ and $\si(A^c) = \bar \eta$ (under $\ell\loc$). 
% The strategy is to condition on $\si(A_r)=\ze$ in addition to $\si(A^c) = \eta$ or $\si(A^c) = \bar \eta$, and look for disconnecting points in the remaining region $A\bs A_r$. If we find disconnecting points on both sides, then the coupling succeeds; otherwise, we are not able to control the configuration conditioned on $\bar \eta$. However, the latter event happens with small probability conditioning on $\si(A_r)=\ze$, which would give a relative error bound. This gives one direction of the inequality, and switching the roles of $\eta,\bar\eta$ gives the other direction. 
\begin{proof}
Independently draw
\[
X\sim \mu(\cdot \mid \sigma(A^c)=\eta, \sigma \in \vec{\ell}\loc), \qquad 
Y^0\sim \mu(\cdot \mid \sigma(A^c)=\bar \eta, \sigma \in \vec{\ell}\loc).
\]
We will define a coupling $(X,Y)$ with $Y\sim \mu(\cdot \mid \sigma(A^c)=\bar \eta, \sigma \in \vec{\ell}\loc)$ such that $X(A_r)=Y(A_r)$ as much as possible, by having $Y$ agree with $Y^0$ outside the outermost common disconnecting points of $(X,Y^0)$ and then replacing $Y^0$ by $X$ inside.

Define the good event
$G$ that there exists both a common disconnecting point for $X$ and $Y^0$ in $[i+1,i+r-\ell_{\max}]$ and a common reverse-disconnecting point in $[n-j-r+\ell_{\max}+1,n-j]$. Note that a buffer of length $\ell_{\max}$ remains between these intervals and $A_r$. 

We first lower-bound the probability of $G$. For the common  disconnecting point on the left ($[i+1,i+r-\ell_{\max}]$), consider $k_s=i+s\ell_{\max}$ for $s=1,\ldots, m$, where $m = \ff{r}{\ell_{\max}}-1$. 
By item (2) of \Cref{lem:prob-of-good-point} applied with boundary conditions on $[k_{s-1}]\cup [i+r+1,n]$, noting that $k_s\le (i+r)-\ell_{\max}$ (and that we can condition further on the values in $[n-j-r+1,n-j]$ to evaluate the probability), for any $\tau$,
\begin{align*}
    \mathbb{P}\ba{
k_s \text{ disconnecting for }X \mid X(A_r) = \tau, X([k_{s-1}]) = h
    }\ge \rh
\end{align*}
for some $\rh = e^{-O(1/\ep)}$. 
Since this is true for $Y^0$ as well, by sequentially looking for common disconnecting points, the probability of no common left disconnecting point for $X$ and $Y^0$ is at most $(1-\rh^2)^m$. The same argument holds for right disconnecting points, so 
\begin{equation}
\label{e:bad-cond-Ar}
\mathbb{P}[G^c \mid X(A_r) = \ze] \le 2(1-\rh^2)^m = :\de(r)
\end{equation}
Note that $\de (r)= 2\exp\pa{-\Om\pf{r}{e^{O(1/\ep)}\ell_{\max}}}= \exp\pa{-\Om\pf{r}{e^{O(1/\ep)}\ell_{\max}}}$ for $r= \Om(e^{C/\ep}\ell_{\max})$, with appropriate $C$.

To finish defining the coupling, let $E_{a,b}$ be the event that the leftmost common disconnecting point for $X$ and $Y^0$ in $[i+1,i+r-\ell_{\max}]$ is $a$ and the rightmost common reverse-disconnecting point in $[n-j-r+\ell_{\max}+1,n-j]$ is $b$. Then $$G = \bigsqcup_{a\in [i+1,i+r-\ell_{\max}], b\in [n-j-r+\ell_{\max}+1,n-j]} E_{a,b}\,.$$ 
On the event $E_{a,b}$, set $Y([a+1,b-1]) = X([a+1,b-1])$ and $Y([a+1,b-1]^c) = Y^0([a+1,b-1]^c)$ as ordered sets. On $G^c$, set $Y=Y^0$.

We will check that $Y$ has the same distribution as $Y^0$. Note that on $E_{a,b}$ the distribution of $X([a+1,b-1]) \mid X([a+1,b-1]^c)$ and 
$Y^0([a+1,b-1]) \mid Y^0([a+1,b-1]^c)$ on $E_{a,b}$ are identical, because the available particles are the same and $[a+1,b-1]$ is an interval. 

For any function $F:S_n\to \R$, conditioning on the two exteriors and noting $E_{a,b}$ is measurable with respect to $X([a+1,b-1]^c), Y^0([a+1,b-1]^c)$
gives
\begin{align*}
    &\E[\one_{E_{a,b}} F(Y) \mid X([a+1,b-1]^c), Y^0([a+1,b-1]^c)]\\
    &= 
    \one_{E_{a,b}}
    \sum_\xi 
    \mathbb{P}[X([a+1,b-1]) = \xi \mid X([a+1,b-1]^c)]
    F(Y|_{[a+1,b-1]^c}\opl \xi) \\
    &= \one_{E_{a,b}}
    \sum_\xi 
    \mathbb{P}[Y^0([a+1,b-1]) = \xi \mid Y^0([a+1,b-1]^c)]
    F(Y^0|_{[a+1,b-1]^c}\opl \xi)\\
    &=     \E[\one_{E_{a,b}} F(Y^0) \mid X([a+1,b-1]^c), Y^0([a+1,b-1]^c)].
\end{align*}
Taking expectations and summing over the disjoint events $E_{a,b}$, and noting that $Y=Y^0$ on $G^c$, we have that $Y$ and $Y^0$ have the same distribution. On $G$, $A_r\sub [a+1,b-1]$, so $Y(A_r)=X(A_r)$. Hence, recalling \eqref{e:bad-cond-Ar},
\begin{align*}
\mu(\si(A_r) = \ze \mid \sigma(A^c)=\bar \eta, \sigma \in \vec{\ell}\loc)&=
\mathbb{P}(Y(A_r) = \ze)\ge 
\mathbb{P}(X(A_r) = \ze, G )\\
&\ge (1-\de(r))\mathbb{P}(X(A_r) = \ze) \\
&= (1-\de(r)) \mu(\si(A_r) = \ze \mid \sigma(A^c)=\eta, \sigma \in \vec{\ell}\loc).
\end{align*}
Reversing the roles of $\eta$ and $\bar\eta$ shows the inequality with $\eta$ and $\bar \eta$ reversed. Hence
\begin{align*}
    1-\de(r) \le \frac{\mu(\sigma(A_r)= \zeta \mid \sigma(A^c)=\eta, \sigma \in \vec{\ell}\loc)}{\mu(\bar \sigma(A_r)= \zeta \mid \bar \sigma(A^c)=\bar \eta, \bar \sigma \in \vec{\ell}\loc)} \le \rc{1-\de(r)},
\end{align*}
giving the desired inequality for appropriate choice of $C'(\ep)$. The cases $i=0$ and $j=0$ follow as in the proof of \Cref{lem:spatial-mixing} by noting that there are boundary conditions on only one side, so it suffices to find a disconnecting point on the opposite side.
\end{proof}

\section{Burn-in via dominating exclusion processes}\label{sec:burn-in}

Our aim in this section is to use dominating ASEP processes to establish the following burn-in type of estimates to get to $\ell_0\loc$ where $\ell_0 = C_0(\ep) \log n$.

Our main result in this section is that for  $\vec{\ell}\in \mathcal L_n$, the adjacent transposition chain $\Mnn$ restricted to $\vec{\ell}\loc$ for $\vec{\ell}\in \mathcal L_n$ reaches $\ell_0{\loc}$ and stays there for arbitrary long polynomial timescales. In particular, this holds for the unrestricted dynamics.

\begin{lemma}\label{lem:single-site-burn-in}
Consider $\Mnn$ with parameters $\vec p$ satisfying $\ep$-positive bias. 
    There exist $    C_0(\varepsilon) = O (1/\eps)$ and 
     {$n_0 = O(\log (1/\eps)/\eps)$}
    such that for all $n\ge n_0$ and any $\vec{\ell}\in \mathcal L_n$, $\Mnn$ restricted to $\vec{\ell}\loc$ is such that for $t\ge C_0(\ep) n^2$ the chain is in $(C_0(\ep)\log{n})\loc$  except with probability $n^{- 10}$. 
\end{lemma}

Note, \Cref{lem:single-site-burn-in} is a more general form of \Cref{lem:single-site-burn-in-simpler}.
We also need a form of this burn-in statement for a block version of the dynamics where in each time-step, $B \in \{B_W,B_E\}$ is chosen randomly, and the configuration $\sigma(B)$ is resampled conditionally on $\sigma(B^c)$.

\begin{lemma}\label{lem:restricted-block-dynamics-burn-in}
There exists $C_1>0$,  and for any $\ep$, there exist $C_0(\eps) = O(1/\eps)$ and $n_0 = O(\log (1/\eps)/\eps)$ such that the following holds for all $\vec{p}$ satisfying $\eps$-positive bias and $\vell\in \mathcal L_n$. 
    The block dynamics restricted to $\vec{\ell}\loc$ is such that for 
    any $n\ge n_0$  and 
    any $t \ge C_1  \log n$,  
    the chain is in $(C_0(\eps)\log{n})\loc$ except with probability $n^{-10}$. 
    \end{lemma}

We prove these burn-in style results by stochastic comparisons of various particle counts to the Asymmetric Exclusion Process (ASEP), for which the mixing time and convergence behavior are well understood.

\subsection{The asymmetric exclusion process}

The $k$-particle asymmetric exclusion process (ASEP) is a Markov chain on the state space of assignments 
$$\binom{[n]}{k}:=\set{Y\in \{0,1\}^n}{\sumo in Y(i)=k}\,,$$ where the value $1$ at a site means a particle is present, and a $0$ means the site is empty. 

We describe the process in discrete time. At each step, the ASEP chain with parameter $q\in (\frac{1}{2},1)$ picks an edge (pair of adjacent sites) $(i,i+1)$ uniformly at random, and exchanges the states on the sites with probabilities 
\begin{align*}
    P((0,1)\to (1,0)) = q \qquad \text{and} \qquad  P((1,0) \to (0,1))  = 1-q\,.
\end{align*}
(Notice that if both sides are occupied or both sides are empty, the interchange does nothing so the probability does not matter.) 
We denote the stationary distribution by $\nu_{[n]}^{(k)}$. 
For an extensive background on the asymmetric exclusion process, see Section III of~\cite{liggett1985interacting}. Throughout this section, we will take $q$ given by the relation 
\begin{align}\label{eq:q-eps-relation}
\frac{q}{1-q} = 1+\eps\,,
\end{align} 
for which the constant-$\eps$-positively-biased adjacent transposition shuffle is directly related to the $q$-ASEP.

\subsection{Coupling the biased adjacent transposition shuffle and ASEP}
In the below, we make some observations on  monotonicity for the ASEP process itself, and monotonicity between the projected $\Mnn$ (see \Cref{def:projection-to-compare-to-ASEP})
and a $k$-particle ASEP. We define a partial order on $\binom{[n]}{k}$ where $\eta \le \eta'$  (read: $\eta$ is to the left of $\eta'$) if for every $r\le n$, $\eta$ has more particles on $[r]$ than $\eta'$ does. Such monotonicity properties are standard in the literature (see e.g., Section~23 of~\cite{LP}); we include proofs for our specific setting for completeness. We begin with an observation about the potential ways such a monotonicity can be violated in one step. 

\begin{observation}\label{obs:ways-to-violate-monotonicity}
    Suppose $Y,Y'\in \binom{[n]}{k}$ with $Y \le Y'$. If they both (possibly) make a swap on an edge $e= (i,i+1)$ to generate $Y_1, Y_1'$, then $Y_1 \le Y'_1$ unless at least one of the following occurred: 
    \begin{enumerate}
        \item $(Y(i),Y(i+1))= (1,0)$ and $(Y'(i),Y'(i+1)) \in \{ (1,0),(0,1)\}$ and $(Y_1(i),Y_1(i+1))=(0,1)$ while $(Y_1'(i),Y_1'(i+1))=(1,0)$. 
        \item $(Y(i),Y(i+1)) \in  \{(0,1),(1,0)\}$ and $(Y'(i),Y'(i+1)) = (0,1)$ and $(Y_1(i),Y_1(i+1))=(0,1)$ while  $(Y_1'(i),Y_1'(i+1))=(1,0)$. 
    \end{enumerate}
\end{observation}

\begin{proof}
We consider the immediate cases first. If both $Y$ and $Y'$ are $(0,0)$ or $(1,1)$ or don't execute the swap at $(i,i+1)$ (regardless of their configuration), then evidently the ordering cannot change. Likewise, if they both are $(1,0)$ and both swap, or both are $(0,1)$ and both swap, the monotonicity cannot change. Finally, if $Y$ is making a $(0,1) \to (1,0)$ swap, or $Y'$ is making a $(1,0)\to (0,1)$ swap, then the ordering is only improving. 

    We now consider the remaining cases that are not in items 1--2 in the observation:
    \begin{enumerate}
        \item $Y$ is making a $(1,0) \to (0,1)$ swap, while $Y'$ is $(0,0)$;
        \item $Y$ is making a $(1,0)\to (0,1)$ swap while $Y'$ is $(1,1)$;
        \item $Y'$ is making a $(0,1)\to (1,0)$ swap while $Y$ is $(0,0)$;
        \item $Y'$ is making a $(0,1)\to (1,0)$ swap while $Y$ is $(1,1)$.
    \end{enumerate}
    For these cases, let $N_i = |\{j\le i: Y(j) = 1\}|$ and define $N_i'$ analogously. The inequality $Y \le Y'$ is equivalent to $N_i \ge N_i'$ for all $i\in [n]$. 
    
    In the first case, since $Y \le Y'$, one has  $N_{i-1} \ge N_{i-1}'$ while $N_{i} = N_{i-1}+1$ and $N_i'= N_{i-1}'$, so $N_i \ge N_{i}'+1$ and so decrementing $N_i$  and leaving $N_i'$ unchanged retains the ordering. 

    In the second case, since $Y \le Y'$, one has $N_{i+1}\ge N_{i+1}'$, while $N_{i+1} = N_i$ and $N_{i+1}' = N_{i}'+1$, so  $N_{i} \ge N_{i}'+1$ and decrementing $N_i$ and leaving $N_i'$ unchanged retains the ordering. 

    In the third case, since $Y\le Y'$, one has $N_{i+1}\ge N_{i+1}'$ but $N_{i+1} = N_i$ while $N_{i+1}' = N_{i}'+1$, so $N_{i} \ge N_{i}'+1$; thus incrementing $N_{i}'$ retains the ordering.  

    In the fourth case, since $Y\le Y'$, one has $N_{i-1}\ge N_{i-1}'$ but $N_{i} = N_{i-1}+1$ while $N_{i}' = N_{i-1}'$, so $N_{i} \ge N_{i}'+1$; thus incrementing $N_{i}'$ retains the ordering.  
\end{proof}

The following lemma states that the ASEP process is monotone in this partial order.

\begin{lemma}[ {Attractivity of the basic coupling}]\label{lem:ASEP-monotonicity}
    If $Y_t$ and $\tilde Y_t$ are $k$-particle ASEPs with $Y_0 \le \tilde Y_0$, then there exists a coupling such that for all $t\ge 0$, one has $Y_t \le \tilde Y_t$. 
\end{lemma}

\begin{proof}
    Consider the following one-step coupling,  {called the basic coupling for exclusion processes.} If $Y_s,Y'_s$ have $k$ particles each and $Y_s \le Y'_s$, then to generate $Y_{s+1}, Y_{s+1}'$  select the same edge $e= (i,i+1)$ to update, and use a single uniform random variable $U\sim \text{Unif}[0,1]$ to generate $(Y_{s+1},Y_{s+1}')$ as follows: if either one's state is $(1,0)$ or $(0,1)$, set it equal to $(1,0)$ if and only if $U\le q$. 

    By Observation~\ref{obs:ways-to-violate-monotonicity}, we only need to consider the possibilities listed in items 1--2 therein as ways that monotonicity could be violated by this update. But the definition of the coupling ensures that if both $(Y_s(i),Y_{s}(i+1))$ and $(Y_s'(i),Y_s'(i+1))$ are in $\{(0,1),(1,0)\}$ then at time $s+1$ they agree on locations $i,i+1$ and so neither case occurs. 
\end{proof}

The next lemma says that projections of an $\eps$-positively biased adjacent transposition shuffle are stochastically to the left of the $q$-ASEP with $q/(1-q) = 1+\eps$. To properly define this stochastic comparison, we define the following projections which track the locations of the first $k$ particles.

\begin{definition}\label{def:projection-to-compare-to-ASEP}
    For each $0< k< n$, 
    define the function $\eta_k:S_n\to \binom{[n]}{k}$ which takes a configuration $\sigma \in S_n$ and outputs $\eta_k(i) = \mathbf 1\{\sigma(i)\le k\}$.
\end{definition}

\begin{lemma}\label{lem:domination-by-ASEP}
    If $\vec{p}$ is $\eps$-positively biased, then the chain $\sigma_t \sim \Mnn$ on $S_n$ is such that for each $k$, the process $\eta_k(\sigma_t)$ is stochastically to the left of the $q$-ASEP process with parameter $q/(1-q) = 1+\eps$, initialized from $\eta_k(\sigma_0)$. 
\end{lemma}

\begin{proof}
    We couple $(\sigma_t)_{t\ge 0}$ to $n-1$ different ASEP process $Y_t^{(1)},...,Y_t^{({n-1})}$ where $Y_t^{(k)}$ is a $k$-particle ASEP initialized from $Y_0^{(k)} = \eta_k(\sigma_0)$. The coupling operates by choosing the same edge to update at each time step, and using the same uniform random variable $U_{t+1}\sim \text{Unif}[0,1]$ to generate the update. The latter statement means that we update as follows: suppose $\sigma_t(i)<\sigma_{t}(i+1)$ and $(Y_{t}^{(k)}(i), Y_{t}^{(k)}(i+1)) \in \{(0,1),(1,0)\}$; then 
    \begin{align*}
        (\sigma_{t+1}(i),\sigma_{t+1}(i+1)) & = (\sigma_t(i),\sigma_t(i+1)) &&\text{if and only if $U_{t+1}\le p_{\sigma_t(i),\sigma_t(i+1)}$}\,, \\ 
        (Y_{t+1}^{(k)}(i), Y_{t+1}^{(k)}(i+1)) & = (1,0) &&\text{if and only if  $U_{t+1}\le q$}\,.
    \end{align*}
    If $\sigma_t(i)>\sigma_{t}(i+1)$ instead, then set 
    \[(\sigma_{t+1}(i),\sigma_{t+1}(i+1)) = (\sigma_{t}(i+1),\sigma_t(i)) \qquad \text{if and only if }
    U_{t+1}\le p_{\sigma_t(i+1),\sigma_t(i)}.\] 

    In this manner, one has the property that if both $(\eta_k(\sigma_{t})(i), \eta_k(\sigma_t)(i+1))$ and $(Y^{(k)}_t(i),Y_t^{(k)}(i+1))$ are in $\{(0,1),(1,0)\}$, then 
    \begin{itemize}
        \item if the former is $(1,0)$, since $p_{\sigma_t(i),\sigma_t(i+1)} \ge q$ by $\eps$-positive bias, 
        \begin{align*}
        (Y^{(k)}_{t+1}(i),Y_{t+1}^{(k)}(i+1))= (1,0) \implies (\eta_k(\sigma_{t+1})(i), \eta_k(\sigma_{t+1})(i+1))= (1,0)\,;
    \end{align*}
    \item if the former is $(0,1)$, since $p_{\sigma_t(i+1),\sigma_t(i)}\ge q$ by $\eps$-positive bias, 
            \begin{align*}
        (\eta_k(\sigma_{t+1})(i), \eta_k(\sigma_{t+1})(i+1))= (0,1) \implies (Y^{(k)}_{t+1}(i),Y_{t+1}^{(k)}(i+1))= (0,1)\,. 
    \end{align*}
    \end{itemize}
        By Observation~\ref{obs:ways-to-violate-monotonicity} and the fact that the monotonicity held at time $t$, the above two items imply it also holds at time $t+1$. 
\end{proof}

Taking $t\to\infty$ in the above lemma, we also deduce the {stochastic} domination relation for the stationary distributions of the first $k$ coordinates in the shuffle and $k$-particles ASEP. 
{We use $\preceq$ to denote stochastic domination with respect to the partial order $\le$.}

\begin{corollary}\label{cor:stationary-domination-by-ASEP}
    Suppose $\vec{p}$ is $\eps$-positively biased. If $\sigma\sim \mu$, then for every $k\le n$, one has $\eta_k(\sigma)\preceq Y^{(k)}_\infty$ where $Y^{(k)}_\infty $ is drawn from the $q$-ASEP stationary distribution with $k$ particles. 
\end{corollary}

The same inequalities hold for the restricted dynamics where we restrict to $\vec{\ell}\loc$ for a vector $\vec{\ell}$ in $\mathcal L_n$. Indeed, that will be a direct consequence of the ASEP's own monotonicity.

\begin{lemma}\label{lem:comparison-to-ASEP-ell-restricted}
     Suppose $\vec{\ell}\loc$ for $\vec{\ell} \in \mathcal L_n$. If $\sigma_t$ is the adjacent transposition chain restricted to $\vec{\ell}\loc$, for any $k$ the processes $\eta_k(\sigma_t)$ are stochastically to the left of the $k$-particle ASEP $Y^{(k)}_t$ started from $\eta_k(\sigma_0)$. 
     It also holds that $\eta_k(\si) \preceq Y^{(k)}$ where $\si$ is drawn from the restricted stationary distribution and $Y^{(k)}$ is drawn from the stationary distribution of $q$-ASEP.
     \end{lemma} 

\begin{proof}
Consider the same coupling as in the proof of \Cref{lem:domination-by-ASEP}, except that rejections override the coupling.
    By Claim~\ref{cl:ooo}, the only moves that get rejected are ones where $\sigma_t(v) = i$ and $\sigma_t(v+1) =j$ for $i<j$. Therefore, in $\eta_k(\sigma_t)$ these vertices can be in states $(0,0)$, $(1,1)$, or $(1,0)$ but not $(0,1)$. 
    % In the coupling to the $k$-particle ASEP, 
    % $(Y_t^{(k)}(v), Y_t^{(k)}(v+1))$ has to be either $(0,0)$, $(1,1)$, or $(1,0)$. 
        In particular a move $(0,1)\to (1,0)$ is never rejected in $\eta_k$. By the monotonicity of the ASEP process per Lemma~\ref{lem:ASEP-monotonicity}, this implies that the restricted dynamics $\eta_k(\sigma_t)$ stays stochastically below the ASEP $Y^{(k)}_t$.  
\end{proof}

\subsection{Analysis of the ASEP} 

The domination of the biased adjacent transposition by $n$ $q$-ASEP processes implies that if the coupled ASEP processes have mixed, their stationary distributions can be used to control the locations of the particles in the shuffle. The paper~\cite{benjamini2005mixing} established that the discrete-time mixing time of $q$-ASEP is $O(n^2)$, and~\cite{LabbeLacoinCutoffASEP} proved cutoff. 

\begin{lemma}[\cite{benjamini2005mixing}, with \cite{LabbeLacoinCutoffASEP} for the $\eps$ dependencies]
\label{lem:ASEP-mixing-time}
    The mixing time of the     $k$-particle $(\frac{1}{2}+\ep')$-ASEP process is at most $\fc{n^2}{\ep'} (1+o(1))$ for every $k\le n$, asymptotically as $n\to \iy$, and its inverse spectral gap is $O\pf{n}{(\ep')^2}$ for all $n$.     \footnote{Note this $\ep'$ is such that $\fc{\rc 2+\ep'}{\rc 2-\ep'} = 1+\ep$. When $\ep,\ep'$ are small, they differ by constant factors.}
\end{lemma}

Given the stochastic comparisons of Lemma~\ref{lem:comparison-to-ASEP-ell-restricted} and the mixing time bound of Lemma~\ref{lem:ASEP-mixing-time}, we can use the stationary distribution of the ASEP to bound the behavior of the biased adjacent transposition shuffle using stationary estimates of the ASEP for certain monotone events. The stationary distribution of the $q$-ASEP is carefully characterized in~\cite{liggett1985interacting};  we need the following tail bound on the right-most particle.   

\begin{lemma}\label{lem:ASEP-stationary-estimate}
        Let $Y^{(k)}$ be drawn from the $k$-particle $q$-ASEP stationary distribution with $q> \rc 2$ as in~\eqref{eq:q-eps-relation} ($\fc{q}{1-q}=1+\ep$). Then for every $k\le n$, for $r \ge C_q$,        \begin{align*}
            \nu^{(k)}_{[n]}( \max \{v: Y^{(k)}(v) =1\} \ge k + r) 
            \le             e^{-\ep' r/4}
        \end{align*}
                                when $r\ge \fc{4}{\ep'}\log \pf{2}{\ep'}$.
        \end{lemma}

\begin{proof}
        Consider the distribution $\bar \mu$ on $S_n$ with 
    \[
p_{i,j} = 
\begin{cases}
    q, &i\le k<j\\
        1, &i,j\in \{1,\ldots, k\} \text{ or } i,j\in \{k+1,\ldots, n\}
\end{cases}\,.
    \]
    Recall the projection $\eta_k:S_n\to \binom{[n]}{k}$ defined in \Cref{def:projection-to-compare-to-ASEP}.
            Because the particles $\{1,\ldots,k\}$ must be in order and $\{k+1,\ldots, n\}$ must also be in order, the preimage of each $Y\in \binom{[n]}{k}$ is a single point, and
    \[(\eta_k)_* \bar\mu(Y)
    =
    {q^{|\set{1\le i<j\le n}{Y(i)=1,Y(j)=0}|}
    \pa{1-q}^{|\set{1\le i<j\le n}{Y(i)=0,Y(j)=1}|}}\,.
    \]
    It is easy to check that this distribution satisfies detailed balance with respect to the ASEP moves, and hence this is the stationary distribution of ASEP.

    Now letting $\ep$ be as in~\eqref{eq:q-eps-relation}, 
    by \Cref{l:stoch-dom-geo} applied to the probability conditioned on specific values of $\si^{-1}(1),\ldots, \si^{-1}(i-1)$ and then averaging,
    \[
\bar \mu(\si^{-1}(i) <k+r\mid
\si^{-1}(1),\ldots, \si^{-1}(i-1)<k+r)
\ge 1-\prc{1+\ep}^{k+r-i},
    \]
    noting that we are asking for the probability that $\si^{-1}(i)$ is one of the first $k+r-i$ available indices.
    Therefore,
    \begin{align*}
        \bar \mu(\si^{-1}(1),\ldots, \si^{-1}(k)<k+r)
        & \ge
        \prod_{i=1}^k \pa{1-\prc{1+\ep}^{k+r-i}}
        \ge
        \prod_{i=r}^{\iy}\pa{1-\prc{1+\ep}^i}\\
        &\ge 1-\sum_{i=r}^{\iy} \prc{1+\ep}^i
        = 1 - \prc{1+\ep}^r \fc{1+\ep}{\ep}
                            \end{align*}
                                                Hence
    \begin{align*}
    \mathbb P( \max \{v: Y^{(k)}(v) =1\} \ge k + r) &\le 
\prc{1+\ep'}^r \fc{1+\ep'}{\ep'}\\
&\le e^{-\ep' r/4}\cdot e^{-\ep' r/4} \fc{1+\ep'}{\ep'}\le e^{-\ep' r/4}\,,
    \end{align*}
        where $\ep'=\min\{\ep,1\}$ and $r\ge \fc{4}{\ep'} \log \pf{2}{\ep'}$.
\end{proof}

Since we also want a burn-in estimate for block dynamics, not just the adjacent transposition chain, we also require a mixing time estimate for $q$-ASEP block dynamics. 

We define a block-update dynamics for the ASEP, and argue that the mixing time of this block dynamics of the ASEP is order $1$ and moreover, that its inverse spectral gap is order~$1$.

Consider the block dynamics with the following blocks (West and East): \begin{align*}
    B_W = \{1,...,\lceil 2n/3\rceil\} \qquad B_E= \{\lfloor n/3\rfloor,...,n\}\,. 
\end{align*}
In each step, randomly pick $B\in\{B_W,B_E\}$ to update, and when block $B$ is selected to update, the configuration on $B$ is replaced by a stationary sample of the $q$-ASEP on $B$ conditional on the configuration on $B^c$. Note that this distribution is simply that of a $q$-ASEP on $B$ with the number of particles equal to $k$ minus the number on $B^c$.

\begin{lemma}\label{lem:ASEP-block-mixing-time}
    Consider the block dynamics for $q$-ASEP on an interval of length $n$ with $k\le n$ particles. There exists a constant $C_1$ independent of $\eps$ such that for all $n \ge C_1\log (1/\eps) /\eps$, the mixing time and inverse spectral gap of the block dynamics is at most $C_1$. 
\end{lemma}

\begin{proof}
    We first consider the case $k\le n/2$. By Lemma~\ref{lem:ASEP-stationary-estimate} and the lower bound on $n$, the stationary distribution of $k$-particle $q$-ASEP on $[n]$ and $k$-particle $q$-ASEP on $B_{W}$ it suffices to couple the chain $Y_{t}^{(k)}$ from any initialization $Y^{(k)}_0$ to that stationary distribution $\nu_{[\ce{2n/3}]}^{(k)}$.

    We claim that so long as $n \ge \frac{C \log (1/\eps)}{\eps}$ for a uniform constant $C$, then the mixing time of the block dynamics is an $\eps$-independent constant. 
    \begin{enumerate}
        \item If the East block is updated first, then by Lemma~\ref{lem:ASEP-stationary-estimate} and the lower bound on $n$, with probability at least $0.99$, the number of particles $k_E$ that remain in $B_E\setminus B_W$ is at most $n/2 - n/3 + \delta n < n/(5.5)$ (for $\delta = 1/100$).
        \item If the West block is updated next, then by Lemma~\ref{lem:ASEP-stationary-estimate} and the lower bound on $n$, with probability at least $0.99$ at most $n/2 - k_E - n/3 + \delta n$ will remain in $B_E \cap B_W$ and therefore at most $n/(5.5) + \delta n <n/5$ many particles will remain in $B_E$. 
        \item If the East block is updated next, by Lemma~\ref{lem:ASEP-stationary-estimate} and the lower bound on $n$, with probability at least $0.99$, all particles will be in $B_W$; 
        \item If the West block is updated next, then the distribution will be exactly equal to the stationary $q$-ASEP distribution conditioned on all particles being in $B_W$.
    \end{enumerate}
    {If these four updates are made and the stated events occur, then because the final distribution is that of the stationary distribution conditioned on all particles being in $B_W$, the two chains can be coupled.}
    The above bounds the mixing time on a segment of length $n \ge (C \log (1/\eps))/\eps$ for a uniform constant $C$ by the minimum time by which with probability $\ge 0.99$, we get the above four consecutive updates. This therefore gives a uniform constant bound on the block mixing time {not depending on $n$} (time {$t=400$} should for instance suffice, {as it suffices for $(1-\prc{2}^4(0.98)^3)^{\fl{t/4}}\le 0.01$}). 

    For $k \ge n/2$, we argue by {the standard particle-hole symmetry}: 
    consider the mapping which flips 0 and 1, so that the particles are following $q$-ASEP where the bias is to the right instead of to the left, and where the number of particles is at most $n/2$. The argument then proceeds exactly symmetrically to the above.
\end{proof}

\subsection{Proofs of burn-in consequences}

We are now able to combine the results of the previous two subsections to deduce the burn-in results for the adjacent transposition and block chains for the biased adjacent transposition shuffle. 

\begin{proof}[\textbf{\emph{Proof of Lemma~\ref{lem:single-site-burn-in}}}]
        By Lemma~\ref{lem:comparison-to-ASEP-ell-restricted}, the $k$th particle of $\sigma_t$ is {stochastically} to the left of the right-most particle of the $k$-particle ASEP, $Y^{(k)}_t$, {i.e. $\max \{v: Y^{(k)}_t(v) = 1\}$ stochastically dominates $\si_t^{-1}(k)$.} By the mixing time and inverse spectral gap bounds of Lemma~\ref{lem:ASEP-mixing-time}, for each $k$, for large enough $n$, for all $t \ge T_1:=C(\ep) n^2 + \frac{r}{\eps^2}  n\log n$ for $C(\eps) = O\pf{1}{\eps}$
                                and constant $r$, 
        one has 
        \begin{align}
        \label{e:mix-time-n-12}
            \max_{y_0}\|\mathbb P_{y_0}(Y^{(k)}_t\in \cdot) - \nu^{(k)}\|_{\tv} \le n^{-12}\,.
        \end{align}
        In particular, combined with Lemma~\ref{lem:ASEP-stationary-estimate} and ensuring $n\ge n_0 = \Omega(\frac{\log (1/\eps)}{\eps})$ and $\ell = \Omega(\frac{\log (1/\eps)}{\eps})$, we have for all $t\ge T_1$ that  
        \begin{align*}
            \mathbb P( \exists k: \max \{v: Y^{(k)}_t(v) = 1\}\ge k+\ell) \le n \big(n^{-12} + e^{ - \eps \ell/4} \big).
        \end{align*}
                By the {aforementioned} stochastic domination, %between $Y_t^{(k)}$ and $\si_t$, 
                this implies 
        \begin{align*}
            \mathbb P( \exists k: \sigma_t^{-1}(k)\ge k+\ell_0) \le n \big(n^{-12} + e^{ - c_\eps \ell_0} \big)
        \end{align*}
        which is at most $n^{-10}/2$ for $\ell_0 = C'(\ep) \log n$ with a large enough constant $C'(\ep) = O\prc{\eps}$. To localize in the other direction, consider ASEP where holes are swapped with particles, and carry out the identical argument symmetrically to also obtain
        \begin{align*}
            \mathbb P( \exists k: \sigma_t^{-1}(k)\le k-\ell_0) \le \frac{1}{2} n^{-10}\,.
        \end{align*}
        Together with a union bound, these two inequalities imply the lemma. 
\end{proof}

\begin{proof}[\textbf{\emph{Proof of Lemma~\ref{lem:restricted-block-dynamics-burn-in}}}]
        By the stochastic domination of the stationary distributions from Lemma~\ref{lem:comparison-to-ASEP-ell-restricted}, the projections of the block dynamics via $\eta_k$ are each stochastically dominated by the $k$-particle ASEP block dynamics. 

    By the mixing time bound of Lemma~\ref{lem:ASEP-block-mixing-time}, there is a constant $C_1$ not depending on $\ep$ and $n_0(\ep) = O(\log (1/\eps)/\eps)$ such that as long as $n \ge n_0(\ep)$, after $C_1 \log n$ many steps of block dynamics, for any $k\le n$, the TV-distance of the ASEP block dynamics is within $n^{-12}$ from its stationary distribution. {This shows \eqref{e:mix-time-n-12} holds for block dynamics.}
    Following the same reasoning as for the single site dynamics---{namely, noting that the proof of Lemma \ref{lem:single-site-burn-in} shows \eqref{e:mix-time-n-12} and the stochastic domination $\max \{v: Y^{(k)}_t(v) = 1\}\succeq \si_t^{-1}(k)$ imply the conclusion, and both conditions also hold for block dynamics}---this implies that 
        \begin{align*}
            \mathbb P( \exists k: \sigma_t^{-1}(k)\notin [k-\ell_0,k+\ell_0]) \le 2 n \big(n^{-12} + e^{ - \eps \ell_0/4}\big) \le n^{-10}\,,
        \end{align*}
        implying the claimed burn-in time. 
\end{proof}

\section{Polynomial mixing time}
\label{sec:poly-mixing}

In this section we prove a more detailed version of the polynomial mixing time bound stated in \Cref{thm:poly}.

\begin{proposition}\label{prop:polynomial-mixing-time}
    There exists $C>0$ such that for all $\ep>0$, there exists $C'(\ep) = e^{\tilde O(1/\eps^2)}$, such that if $\vec{\ell}\in \mathcal L_n$, then $\Mnn$ restricted to $\vec{\ell}\loc$ has inverse spectral gap at most $C'(\ep) n^C$.
\end{proposition}

As a corollary, by taking $\vec{\ell} = (-\infty,\infty)_i$, which is evidently admissible, this gives polynomial inverse gap of the original unrestricted chain as stated in \Cref{thm:poly}.

The proof of \Cref{prop:polynomial-mixing-time} combines the burn-in estimates of Section~\ref{sec:burn-in} with the spatial mixing estimates amongst burnt-in configurations from Section~\ref{sec:spatial-mixing} to show that the adjacent transposition chain restricted to any admissible $\vec{\ell}\loc$ has polynomial mixing time. 

\subsection{Block dynamics}

The block dynamics is a standard tool used in spin system dynamics for upper bounding mixing times using spatial self-similar structure of the process. It is typically phrased in the context of spin system dynamics.  In the case where $p_{ij}\equiv q$ are uniform for $i<j$, the chain and corresponding exclusion processes admit equivalent height function representations and spin system tools like censoring and block dynamics readily apply. In our context, while the former does not due to the absence of monotonicity, we can still leverage the block dynamics spectral decomposition.

Let us first define what we mean by block dynamics for the biased adjacent transposition shuffle dynamics. In this subsection we are using $N$ to denote an arbitrary length of the 1D chain. 

\begin{definition}\label{def:block-dynamics}
Let $\cB=\{B_1,\dots,B_\ell\}$ denote any collection of subsets of $[N]$ where $\bigcup_i B_i=[N]$.  
The heat-bath block dynamics 
is the {discrete-time} Markov chain which, from a configuration $X_t$, chooses a block $B_i$ with probability $|B_i|/(\sum_i |B_i|)$ from $\cB$, and resamples the entire configuration $X_t(B_i)$ 
according to the stationary distribution conditioned on $X_{t}(B_i^c)$.
\end{definition}

In the context of the biased adjacent transposition shuffle and its stationary distribution $\mu$, as long as 
the edge-sets of the blocks cover the edges of $[N]$, then the block dynamics is ergodic and reversible with respect to $\mu$. Ergodicity follows since the block dynamics can simulate any adjacent transposition,  and reversibility is straightforward from the definition of heat-bath. 

The block dynamics are important for analyzing spin systems because they give us the ability to upper bound the inverse gap of the adjacent transposition chain on $[N]$ by the worst inverse gap of the adjacent transposition chain on a block, times the inverse gap of the block dynamics chain. 
More precisely, the following decomposition result is standard in the analysis of the Glauber dynamics for spin systems, see e.g., \cite[Proposition 3.4]{Martinelli}. 
Our setting does not exactly fit into these spin system setting because two vertices update their assignments simultaneously, but the proof is essentially the same. We include a proof for completeness.

    \begin{proposition}\label{prop:block-dynamics}
    Let $\chi(\cB) = \max_{j\in [N-1]} |\{i: (j,j+1) \in E(B_i)\}|$ denote the maximum coverage
    of an edge in $[N]$.
    For any collection of blocks $\cB$, we have the following:
    \begin{align*}
        \gamma(\Mnn([N]))\ge \chi(\cB)^{-1}  \cdot \gamma(\mathcal M_\cB) \cdot \min_{\eta, i} \gamma(\Mnn(B_i^\eta))\,,
    \end{align*}
    where $\gamma(\Mnn([N]))$ is the spectral gap of $\Mnn$ on $[N]$, $\gamma(\mathcal M_\cB)$ is the spectral gap of the  block dynamics defined by $\cB$, and $\gamma(\Mnn(B_i^\eta))$ is the spectral gap of $\Mnn$ restricted to the set $B_i$ with configuration $\eta$ on $B_i^c=[N]\setminus B_i$. 
\end{proposition}

\begin{proof}
    Let us begin with some notation. For a function $f: S_N \to \mathbb R$, we use $\text{Var}(f)$ to denote the variance with respect to $\mu$. The Dirichlet form for $\Mnn$ is 
    \begin{align}\label{eq:standard-Dirichlet-form}
        \mathcal E(f,f) = \frac{1}{2} \frac{1}{N-1} \sum_{\sigma} \mu(\sigma) \sum_{v\in [N-1]} p_{\sigma(v+1),\sigma(v)}  (f(\sigma) - f(\sigma \circ (v,v+1)))^2\,.
    \end{align}
    Similarly, by definition of the block dynamics $\mathcal B$, its Dirichlet form is given by
    \begin{align}\label{eq:block-Dirichlet-form}
        \mathcal E_{\mathcal B}(f,f) = \frac{1}{2} \sum_{\sigma} \mu(\sigma) \sum_{i\le |\mathcal B|} \frac{|B_i|}{\sum_i |B_i|}\sum_{\sigma': \sigma(B_i^c) = \sigma'(B_i^c)} \mu_{B_i}^\sigma (\sigma') (f(\sigma) - f(\sigma'))^2\,,
    \end{align}
    where $\mu_{B_i}^\sigma$ denotes the distribution on $B_i$ conditional on equaling $\sigma$ on $B_i^c$. Finally, we define the Dirichlet form for the chain $\Mnn(B_i^\sigma)$, i.e., adjacent transposition chain only making updates in $B_i$, given the configuration $\sigma(B_i^c)$: 
    \begin{align}\label{eq:restricted-Mnn-Dirichlet-form}
        \mathcal E_{B_i}^\sigma(f,f) = \frac{1}{2} \sum_{\sigma': \sigma'(B_i^c) = \sigma(B_i^c)}\mu_{B_i}^\sigma(\sigma') \frac{1}{|B_i|} \sum_{(v,v+1)\in E(B_i)} p_{\sigma'(v+1),\sigma'(v)}(f(\sigma') - f(\sigma' \circ (v,v+1)))^2\,.
    \end{align}
    By the Markov property for the distribution, $\mu(\sigma) = \mu(\sigma(B_i^c))\mu^\sigma_{B_i}(\sigma(B_i))$, so we can rewrite $\mathcal E_{\mathcal B}(f,f)$ as 
    \begin{align*}
       \frac{1}{2} \sum_i \sum_{\sigma(B_i^c)= \sigma'(B_i^c)} \mu(\sigma(B_i^c)) & \sum_{\sigma(B_i),\sigma'(B_i)}\frac{|B_i|}{\sum_j |B_j|} \mu^\sigma_{B_i} (\sigma(B_i)) \mu^{\sigma}_{B_i}(\sigma'(B_i)) (f(\sigma ) - f(\sigma'))^2 \\ 
       & = \sum_i  \frac{|B_i|}{\sum_j |B_j|}\sum_{\sigma(B_i^c)} \mu(\sigma(B_i^c)) \text{Var}_{B_i}^\sigma(f)\,,
    \end{align*}
    where $\text{Var}_{B_i}^\sigma$ denotes variance with respect to $\mu_{B_i}^\sigma$
    Plugging in the spectral gap of $B_i$ with boundary conditions $\sigma(B_i^c)$, we get 
    \begin{align}\label{eq:block-dynamics-dirichlet-form-bound}
        \mathcal{E}_{\mathcal B}(f,f) \le \sum_i \frac{|B_i|}{\sum_j |B_j|}\sum_{\sigma(B_i^c)} \mu(\sigma(B_i^c)) \frac{\mathcal E_{B_i}^\sigma(f,f)}{\gamma(\Mnn(B_i^\sigma))}\,.
    \end{align}
    At the same time, for every value of $\sigma(B_i^c)$, the restricted Dirichlet form can be written as 
    \begin{align*}
        \mathcal E_{B_i}^\sigma(f,f)   = \frac{1}{2} \sum_{\sigma'(B_i)} \mu^\sigma_{B_i}(\sigma'(B_i))\frac{1}{|B_i|}\sum_{(v,v+1)\in E(B_i)} p_{\si'(v+1),\si'(v)} \Big( & f(\sigma \opl \sigma') \\ 
        & - f(\sigma \opl \sigma' \circ (v,v+1))\Big)^2\,,
    \end{align*}
    where  
    \[ (\si\opl \si')(x) = \begin{cases}
        \si(x), &x\nin B_i\\
        \si'(x), &x\in B_i.
    \end{cases}
    \]
    Plugging this bound into~\eqref{eq:block-dynamics-dirichlet-form-bound}, and using the Markov property $\mu(\sigma\oplus \sigma') = \mu(\sigma(B_i^c))\mu_{B_i}^\sigma(\sigma'(B_i))$ again,  
   \begin{multline*}     
        \mathcal E_{\mathcal B}(f,f)  \le 
        \\ 
\max_{i,\sigma(B_i^c)} \frac{1}{\gamma(\Mnn(B_i^\sigma))}\cdot\frac{1}{2}\sum_{i} \frac{1}{\sum_j |B_j|}\sum_{\sigma} \mu(\sigma)  \!\!\!\!\!\sum_{(v,v+1)\in E(B_i)} \!\!\!p_{\sigma(v+1),\sigma(v)}(f(\sigma) - f(\sigma \circ (v,v+1)))^2\,.
     \end{multline*}
    Canceling out the $|B_i|$ term, using the lower bound $\sum_i |B_i| \ge N-1$, and noting that the sums over $i$ and $(v,v+1)$ overcount edges $(v,v+1)$ at most $\chi = \chi(\mathcal B)$ many times, we get 
    \begin{align*}
        \mathcal E_{\mathcal B}(f,f) & \le \chi  \cdot \Big( \frac{1}{2} \frac{1}{N-1} \sum_{\sigma} \mu(\sigma) \sum_{v\in [N-1]} p_{\sigma(v+1),\sigma(v)} (f(\sigma)- f( \sigma\circ (v,v+1)))^2\Big)\\
        & \qquad \,\cdot \max_{i,\sigma(B_i^c)} \gamma^{-1}(\Mnn(B_i^\sigma)) \\ 
        & = \chi \cdot \mathcal E(f,f)  \cdot \max_{i,\sigma(B_i^c)}\gamma^{-1}(\Mnn(B_i^\sigma))\,.
    \end{align*}
    Using $\text{Var}(f)  \le \mathcal E_{\mathcal B}(f,f) / \gamma(P_{\mathcal B})$, we get 
    \begin{align*}
        \frac{\text{Var}(f)}{\mathcal E(f,f)} \le \chi \cdot  \gamma^{-1}(\mathcal{M}_{\mathcal B}) \cdot \max_{i,\sigma(B_i^c)} \gamma^{-1}(\Mnn(B_i^\sigma))\,,
    \end{align*}
    which gives the claimed bound on $\gamma(\Mnn)$ (after inverting everything). 
\end{proof}

\begin{rem}\label{rem:restricted-block-dynamics}
Just as the adjacent transposition chain could be restricted to $\vec{\ell}\loc$ per Definition~\ref{defn:localized}, the block dynamics can be restricted to $\vec{\ell}\loc$ by rejecting any update that takes it out of the set. This restricted block dynamics remains ergodic (as it can simulate any move of the restricted adjacent transposition chain, which was ergodic by Claim~\ref{cl:ooo}), and reversible with respect to the conditional stationary distribution $\mu(\cdot \mid \vec{\ell}\loc)$. 

In particular, if we denote by $\tilde \mu = \mu(\cdot \mid \vec{\ell}\loc)$, then for $\sigma \in \vec{\ell}\loc$, the transition rates for the block dynamics are still $\mu( \sigma(B_i)\mid \sigma(B_i^c),\vec{\ell}\loc) = \tilde \mu_{B_i}^\sigma(\sigma(B_i))$ and the Markov property $\tilde \mu(\sigma) =  \tilde\mu(\sigma(B_i^c)) \tilde \mu_{B_i}^{\sigma}(\sigma(B_i))$ still holds. The transition rates for the single-edge dynamics are the same $p_{\sigma(v+1),\sigma(v)}$ so long as the transposition keeps the configuration in $\vec{\ell}\loc$, and are zero otherwise.  
Thus, going through the above proof, up to changing $\mu$ to $\tilde \mu$, the changes are that the sums over $v$ in~\eqref{eq:standard-Dirichlet-form} and~\eqref{eq:restricted-Mnn-Dirichlet-form} are only over those $v$ such that after the transposition the configuration $\sigma$ or $\sigma'$ is still in $\vec{\ell}\loc$. Those constraints match up with one another in the displays following~\eqref{eq:block-dynamics-dirichlet-form-bound} so the rest of the proof is unchanged. Therefore, Proposition~\ref{prop:block-dynamics} also holds to compare the spectral gap of the restriction of $\Mnn$ to $\vec{\ell}\loc$ to the spectral gap of block dynamics restricted to $\vec{\ell}\loc$.  
\end{rem}

\subsection{Spatial recursion}
\label{sub:spatial-recursion}

We now describe the main spatial recursion that gives us the polynomial mixing time bound. Consider the block dynamics $\Mwe$ with the following blocks: 
\begin{align*}
B_{W}= \{1,...,\lceil 2N/3\rceil \}\,,\qquad B_E = \{\lfloor N/3\rfloor ,...,N\}\,.
    \end{align*}

We first establish fast mixing of the block dynamics.  The following result is a generalization of \Cref{lem:block-mixing-time-simpler}.

\begin{lemma}\label{lem:block-mixing-time}
    There exist constants $C_{\mathsf{block}}>0$ (independent of $\eps$), $C_{\mathsf{block}}'(\eps) = e^{O(1/\eps)}$, $n_0 = O (\log(1/\eps)/\eps)$ and $n_1 = e^{O(1/\eps)}$, such that 
    \begin{itemize}
        \item for all $N\ge n_1$, the inverse gap of the block dynamics restricted to $\vec{\ell}\loc$ for $\vec{\ell}\in \mathcal L_N$ is at most $C_{\mathsf{block}}$;
        \item for all $N\in [n_0,n_1]$, the inverse gap of the block dynamics restricted to $\vec{\ell}\loc$ for $\vec{\ell}\in \mathcal L_N$ is at most $C_{\mathsf{block}}'(\eps)$.
    \end{itemize} 
\end{lemma}

\begin{proof}
   To show the block dynamics has $C(\ep)$ spectral gap for some $\ep$, it is sufficient  to show that in time $O(C(\ep) \log N)$, the distance to stationarity is at most $N^{-8}$. This is because the exponential decay rate to stationarity is bounded by the spectral gap: see~\cite[Lemma 20.11]{LP}. 

    By Lemma~\ref{lem:restricted-block-dynamics-burn-in}, the restricted block dynamics $\hat X_t$ burns in to $\ell_0 \loc$ where $\ell_0 = C_0(\ep) \log N$ for $C_0(\ep) = O(1/\eps)$: there exists     $C_1$ 
    such that for any $t\ge  C_1  \log N$,
    \begin{align*}
        \mathbb P(\hat X_{t} \notin \ell_0\loc) \le N^{ - 10}\,.
    \end{align*}
    Let $C'(\ep)=\Te(1)$ for $N\ge n_1$, and $C'(\ep)=e^{\Te(1/\ep)}$ for $N\in [n_0,n_1)$ (to match the expressions for $C_{\mathsf{block}}$ and $C_{\mathsf{block}}'(\eps)$). 
    Then by a union bound, the above holds for all times between $C'(\ep) \log N$ and $2C'(\ep) \log N$: 
    \begin{align*}
        \mathbb P(\exists t\in [C'(\eps) \log N, 2C'(\ep)\log N]: \hat X_{t} \notin \ell_0\loc) \le N^{ - 9}\,.
    \end{align*}
    We now bound the coupling time of two restricted block dynamics chains $\hat X_t, \hat X_t'$  from two different initializations. We use the same choices of blocks for the two, and work on the event that both of them are in $\ell_0\loc$ for all $t\in [C'(\eps) \log N, 2C'(\eps) \log N]$. 
    
    Now consider the event that there is an East block update, immediately followed by a West block update. This happens at least $\delta^{-1} \log N$ many times in the time-frame $[C'(\eps) \log N,2C'(\eps) \log N]$ with probability at least $1-N^{-10}$, as long as the $\eps$-independent constant in $C'(\eps)$ is large enough (as a function of $\delta$). On that event, intersected with the above event that $\hat X_t, \hat X_t'\in \ell_0 \loc$ for all such times, 
    \begin{itemize}
        \item When $B_E$ updates, by Lemma~\ref{lem:spatial-mixing}, as long as $r \ge 10 C_0(\ep) \log N$, there exists a coupling of the resampling step such that with probability at least $1- e^{ - r/(C''(\ep) C_0(\ep)\log N)}$ the two configurations agree fully on $B_E\setminus B_W$, where $C''(\ep) = \exp(O\prc{\ep})$. 
        \begin{itemize}
            \item For $N\ge n_1$, with appropriate $n_1=e^{O(1/\ep)}$, this probability is $\Om(1)$.
            \item 
            For $n_0\le  N\le n_1$, with an appropriate choice of $n_0=O(\log(1/\ep)/\ep)$, we have $10C_0(\ep)\log N\le N\le n_1$, so this probability is $1-\exp(-e^{-O(1/\ep)}) = \Om(e^{-O(1/\ep)})$.
        \end{itemize}
        \item On that event, when $B_W$ updates in the next step, the induced distributions by the two chains on $B_W$ given the configurations on $B_E \setminus B_W$ are identical and therefore they are coupled to equal using the identity coupling. 
    \end{itemize}
    Hence for $N\ge n_1$, with probability at least $1-N^{-9}$, the two chains couple after time $O(\log N)$. For $n_0\le N\le n_1$, with probability at least $1-N^{-9}$, the two chains couple after time $\exp(O(1/\ep))\log N$. Choose constants such that these are bounded by $C'(\ep)\log N$.
    This bounds the mixing time to $N^{-9}$ and hence finishes the proof.
\end{proof}

\subsection{Proof of the polynomial bound}

We are now in position to recurse using Lemma~\ref{lem:block-mixing-time} and apply a crude bound for the base case bound to deduce Proposition~\ref{prop:polynomial-mixing-time}. 

\begin{proof}[\textbf{\emph{Proof of Proposition~\ref{prop:polynomial-mixing-time}}}]
We make the following more precise statement which we prove by induction: There exists a universal $C$, and for every $\eps>0$ there exist constants that depend on $\eps$ as $C'(\varepsilon) = e^{\tilde O(1/\eps^2)}$, $n_0 = O(\log(1/\eps)/\eps)$ and $n_1 = e^{O(1/\eps)}$ such that the following holds for all $N$. 
For all $\Mnn$ on $[N]$ with  $\eps$-positively biased $\vec{p}$, and restricted to any $\vec{\ell}^{(N)}\loc$ for $\vec{\ell}^{(N)}\in \mathcal L_N$, the spectral gap of $\Mnn$ is bounded as follows:
\begin{enumerate}
    \item For $n_0 \le N\le n_1$, we have $\ga(\Mnn)^{-1}\le C'(\ep)N^{C/\ep + C}$.
    \item For $N\ge n_1$, we have $\ga(\Mnn)^{-1}\le C'(\ep) N^{C}$.
\end{enumerate}
    
    \medskip
    \noindent \emph{Base case}. 
    Let $n_0= O(\log (1/\eps)/\eps)$ be as in the previous lemmas. Then as base cases, for all $N\le 2n_0$, we give a crude bound on the mixing time.     We claim that we can couple any two initializations in time $[0,4n_0^2]$ with probability  at least $\exp(-\tilde \Omega(n_0^2))$. Indeed, suppose that the first $N(N-1)$ many updates are exactly $N$ ordered sweeps from right-to-left (this has probability at least $(\frac{1}{N})^{N^2} = e^{ - \tilde O(n_0^2)}$). Furthermore, suppose that every time an edge is selected to be updated, it swaps the particles at its end points if and only if they are out of order. This has probability at least $\frac{1}{2}$ in each step, and therefore $\pf{1}{2}^{N^2}$, conditional on the sequence of updates. The resulting configuration at time $N^2$ will be exactly the ground state $(1,2,...,N)$ in both chains, and they will be thereafter coupled. Since this coupling can be reattempted every $N^2$ many steps, we get a coupling time, and therefore mixing time bound, of $O(N^2) N^{N^2} 2^{N^2} = e^{\tilde O(n_0^2)}$. 

    \medskip
    \noindent \emph{Inductive step}. 
    By \Cref{prop:block-dynamics}, we can bound the inverse spectral gap in terms of the inverse spectral gap of the block dynamics $\mathcal M_\cB$ and of the restriction:
    \begin{align*}
        \gamma(\Mnn)^{-1} \le  2 \ga(\mathcal M_\cB)^{-1}\max_{\xi} \max_{B\in\{ B_W,B_E\}} \gamma(\Mnn(B^{\xi}))^{-1},
    \end{align*}
    where the $2$ comes from $\chi$, and where $\xi$ is a ``boundary condition" meaning a configuration on $[N]\setminus B$. Since $B^{\xi}$ is itself a segment of length $\lceil2N/3\rceil$ with some arbitrary boundary conditions and alphabet available to it, 
    and with $\ep$ still serving as a lower bound on the bias, we can apply Lemma~\ref{l:relabeling} to relabel it so that it is equivalent to an $\eps$-positively biased adjacent transposition chain on $[\ce{2N/3}]$ restricted to some $\vec{\ell}' \in \mathcal L_{\lceil2N/3\rceil}$. 
    Letting $G_N$ be the maximum possible value of $\ga(\Mnn)^{-1}$ for $\ep$-positively biased $\vec p$ and restricted to any ${\vec\ell}^{(N)}\loc$ for ${\vec\ell}^{(N)}\in \cal L_N$. Then substituting in the bounds from \Cref{lem:block-mixing-time}, 
    \begin{align*}
        G_N \le \begin{cases}
            2C_{\mathsf{block}}'(\ep)G_{\ce{2N/3}}, & N \ge n_0\,, \\
            2C_{\mathsf{block}} G_{\ce{2N/3}}, & N\ge n_1\,.
        \end{cases}
    \end{align*}
    for a constant $C_{\mathsf{block}}$ independent of $\eps$ and a constant $C'_{\mathsf{block}}(\eps)$ scaling as $e^{O(1/\eps)}$. 
    The claimed bound on $G_N$ holds for $N\in [n_0,2n_0]$ by the base case analysis. For $2n_0 \le N \le 2n_1$, the induction assumption (1) on $G_{\lceil 2N/3\rceil}$ gives that 
    \begin{align*}
G_{N} \le 
2 C_{\mathsf{block}}'(\ep)C'(\ep) \ce{2N/3}^{C_2/\ep+C}
 & \le 2 C'_{\mathsf{block}} C'(\ep) \pf{\ce{2N/3}}{N}^{C_2/\ep + C} N^{C_2/\ep+C} \\
 & \le C'(\ep) N^{C_2/\ep+C}\,,
    \end{align*}
    so long as the constants $C, C_2$ were sufficiently large (though independent of $N,\eps$), so that $(2/3)^{C_2/\eps}  \ll 1/C'_{\mathsf{block}}$; this holds for large $C_2$ because $C'_{\mathsf{block}}= e^{O(1/\eps)}$.
    
    The above additionally gives that the statement (2) holds for $N\in [n_1,2n_1]$ possibly for a different $C'(\ep)$, still of the same order, by absorbing $(2n_1)^{C_2/\ep+C}$ into the constant. Then, for $N\ge 2n_1$, the induction step is
    \[
G_{N} \le 2C_{\mathsf{block}} C'(\ep) \ce{2N/3}^C
\le 2C_{\mathsf{block}} C'(\ep)\pf{\ce{2N/3}}{N}^{C} N^C
\le C'(\ep) N^C\,,
    \]
    again, so long as $C$ was sufficiently large (independent of $N,\eps$). 
    \end{proof}

As a corollary of Proposition~\ref{prop:polynomial-mixing-time}, as long as we  have an $\exp( -\text{poly}(n))$ lower bound on $\mu_{\min}$, we get a polynomial mixing time and log-Sobolev constant as well. (We will later truncate $p_{ij}$ that are too close to $0$ or $1$ and apply the below.) 

\begin{corollary}\label{cor:polynomial-LSI-and-mixing}
    Suppose that $p_{\min} = \min_{i\ne j} p_{i,j}>0$. Then in the context of Proposition~\ref{prop:polynomial-mixing-time}, $\Mnn$ restricted to $\vec{\ell}\loc$ has mixing time and inverse log-Sobolev constant at most $C'(\eps) n^C \cdot O(n^2 \log \frac{1}{p_{\min}})$
\end{corollary}
\begin{proof}
    The mixing time bound follows from the standard comparison between inverse spectral gap and mixing time~\cite[Eq.~(12.10)]{LP}, up to a factor of $\log \frac{1}{\mu_{\min}}$ where $\mu_{\min} = \min_{\sigma} \mu(\sigma)$. The trivial bound of $\mu_{\min} \ge \frac{1}{n!} p_{\min}^{\binom{n}{2}} \ge \exp( -O(n^2 \log \frac{1}{p_{\min}}))$ gives the claim. The log-Sobolev bound is by a similar standard comparison with factor $\log\frac{1}{\mu_{\min}}$: see e.g.,~\cite[Corollary A.4]{diaconis1996logarithmic}.
\end{proof}

\section{Boosting to sharp mixing time bound}

In this section, we boost the polynomial bound of Proposition~\ref{prop:polynomial-mixing-time} to the optimal $O(n^2)$ mixing time, and prove Theorem~\ref{thm:main}.

\subsection{Fast mixing within the localized set}
Because of Lemma~\ref{lem:single-site-burn-in}, we know that in $C'(\ep) n^2 + O(n \log n)$ time, the adjacent transposition process is burnt-in and gets to a localized configuration, and stays there for large polynomial times. We then need to show that the mixing time restricted to $\ell_0\loc$ is $o(n^2)$. In order to apply Corollary~\ref{cor:polynomial-LSI-and-mixing}, let us assume for now that $p_{\min} = \min_{i\ne j} p_{ij} \ge n^{-\kappa}$ for a sufficiently large  constant $\kappa$; a simple coupling argument will later show that $O(n^2)$ mixing for truncated $p_{ij}$'s satisfying $p_{\min} \ge n^{-\kappa}$ implies $O(n^2)$ mixing for the non-truncated ones.

Consider the restricted Markov chain $\overline{\si}_t$ which rejects any updates that take the configuration outside $\ell_0\loc$ for $\ell_0 = C_0(\ep) \log n$. Let $\bar \mu = \mu(\cdot \mid \ell_0 \loc)$ be its stationary distribution.

\begin{lemma}\label{lem:fast-mixing-within-ell-loc}
    Suppose $p_{\min} \ge n^{-\kappa}$ for a large universal constant $\kappa$. There exist $C(\ep) = e^{\tilde O(1/\eps^2)}$, $n_0(\ep) = O(\log (1/\eps)/\eps)$, $C_0(\ep) = O(1/\eps)$ and a constant $C_1>0$ independent of $\eps$ such that the following holds. For all $n \ge n_0(\ep)$, for $\ell_0=C_0(\ep)\log n$, the $\ell_0\loc$ restricted chain $\overline{\si}_t$ has mixing time at most $C(\ep) n (\log n)^{C_1}$. 
\end{lemma}

\begin{proof}
    We prove this by reducing to a block dynamics restricted to the set $\ell_0\loc$ which we will now denote by $\bar \Omega$. 
    Let $M = e^{\Te(1/\ep)}(\log n)^3$ with appropriate constants to be specified implicitly later in the proof.
    The main reduction will be to the restricted block dynamics with the following two blocks: 
    \begin{enumerate}
\item $B_{0}=\bigcup_{i\geq 0} (6iM, (6i+4)M]$;
\item $B_{1}=\bigcup_{i\geq 0} ((6i+3)M, (6i+7)M]$.
\end{enumerate}
Since comparison from inverse spectral gap to mixing time incurs a factor of $\log \frac{1}{\mu_{\min}}$, which is a factor of $n$ even within $\vec{\ell}\loc$, we will not directly study the block dynamics with these blocks, but will approximately factorize entropy across them, essentially performing a log-Sobolev version of Proposition~\ref{prop:block-dynamics} using the ratio spatial mixing of Lemma~\ref{lem:ratio-spatial-mixing}.

Observe that except at the ends, for the $i$-th segment of $B_0$, its left-most and right-most stretches of length $M$ are subsets of $B_1$. 

\medskip
\noindent \emph{Approximate entropy factorization across the blocks}.
Our aim in this step of the proof is to use the spatial mixing estimates of Section~\ref{sec:spatial-mixing}, specifically the ratio spatial mixing Lemma~\ref{lem:ratio-spatial-mixing} to approximately factor the entropy across $B_0,B_1$. (This is roughly saying that the block dynamics with blocks $B_0,B_1$ satisfies a constant log-Sobolev inequality.) Namely, for all $f:\{\sigma:\sigma\in \ell_0\loc\}\to \mathbb R_{\ge 0}$, 
\begin{align}
    \text{Ent}(f) \le 2\big( \bar \mu [\text{Ent}_{B_0}(f)] + \bar\mu[\text{Ent}_{B_1}(f)]\big)
\end{align}
where $\text{Ent}(f) = \bar \mu [f \log f] - \bar \mu [f]\log \bar \mu[f]$, and the subscript $B_i$ indicates the expectations are taken further conditioned on $\sigma(B_i^c)$. By Proposition 2.1 of~\cite{Cesi-LSI}, it is sufficient to show for $\delta<1/200$, for all functions $g$, 
\begin{align}\label{eq:NTS-by-Cesi}
    \max_{\sigma(B_1^c)\in \ell_0 \loc} |\bar \mu\big[\bar \mu [ g\mid \sigma(B_0^c)] \mid \sigma(B_1^c)\big] - \bar \mu[g]| \le \delta \bar \mu [|g|]\,.
\end{align}
The above~\eqref{eq:NTS-by-Cesi} will follow directly from Lemma~\ref{lem:ratio-spatial-mixing}. Indeed, for each constituent length-$2M$ interval $I_i= ((6i+4)M, (6i + 6)M]$ of $B_1\setminus B_0$, let $\nu^\eta_{I_i}$ denote the marginal on $I_i$ given $\eta$ and $\ell_0\loc$, and observe that 
\begin{align*}
    \bar \mu ( \sigma(B_1 \setminus B_0) \in \cdot \mid \sigma(B_1^c)= \eta) = \bigotimes_{i} \nu_{I_i}^\eta =: \nu^\eta
\end{align*}
as the configuration on $I_i$ is independent of the configuration on all the constituent segments of $B_1$ except the one containing $I_i$ due to the event $\ell_0\loc$. By Lemma~\ref{lem:ratio-spatial-mixing} applied repeatedly, 
\begin{align*}
    \max_{\eta,\eta'\in \ell_0 \loc} \Big\| \frac{\nu^\eta}{\nu^{\eta'}} - 1\Big\|_\infty \le 4(n/M) e^{ - M/(e^{O(1/\eps)}\ell_0)} \le n^{-10}
\end{align*}
so long as the implicit constants in $M = e^{\Theta(1/\eps)} (\log n)^3$ were suitable. The $\|\cdot \|_\infty$ norm here is over all $\ell_0\loc$ configurations on $B_1 \setminus B_0$, which due to the $M \ge \ell_0$ separation of that set from $B_1^c$ where $\eta,\eta'$ live, is not affected by the choice of $\ell_0\loc$ configuration on $B_1^c$ (the set of particles available to each are disjoint).  

Therefore, for $\eta \in \ell_0 \loc$, for $h = \bar \mu[g \mid \sigma(B_0^c)]$ (a function of the configuration on $B_0^c = B_1 \setminus B_0= \bigcup_i I_i$), 
\begin{align}\label{eq:nu-eta-nu-eta-prime}
    |\nu^\eta[h] - \nu^{\eta'}[h]| \le \sum_{\zeta} | h(\zeta)| |\nu^\eta(\zeta) - \nu^{\eta'}(\zeta)| \le n^{-10} \nu^{\eta'}[|h|]\,,
\end{align}
where the sum runs over configurations $\zeta$ on $B_1 \setminus B_0$ that are $\ell_0\loc$. Taking a $\bar \mu$ expectation over $\eta'$ in the above, and by tower law of expectation, noting that $\bar \mu [ \nu^{\eta'}[h]] = \bar \mu[h] =\bar \mu[g]$, the left-hand side of~\eqref{eq:nu-eta-nu-eta-prime} is at least the left-hand side of~\eqref{eq:NTS-by-Cesi}. Using that $\bar \mu[\nu^{\eta'}[|h|]] \le \bar\mu[|g|]$ by Jensen's inequality, the right-hand side is the right-hand side of~\eqref{eq:NTS-by-Cesi} for $\delta = n^{-10}$ which is less than $1/200$ for large $n$.

    \medskip
    \noindent 
    \emph{LSI for each block}. 
    As argued above, the measures $\bar \mu( \cdot \mid \sigma(B_i^c))$ are product distributions and therefore the entropies $\text{Ent}_{B_i}(f)$ factorize into $\text{Ent}_{J_{i,j}}(f)$ where $J_{0,j} = (6jM,(6j+4)M]$ and $J_{1,j} = ((6j+3)M,(6j+7)M]$ and these are entropies with respect to $\bar \mu(\cdot \mid \sigma(J_{ij}^c))$: 
\begin{align}
    \text{Ent}(f) \le 2\big(\sum_{j} \bar \mu [\text{Ent}_{J_{0,j}}(f)] + \sum_{j} \bar\mu[\text{Ent}_{J_{1,j}}(f)]\big)
\end{align}
    We now apply Corollary~\ref{cor:polynomial-LSI-and-mixing} to bound each of the conditional entropies above by the Dirichlet form of the adjacent transposition shuffle on any such segment $J_{ij}$ conditioned on the particles on $J_{ij}^c$ and restricted to $\ell_0\loc$, which we can denote by $\mathcal E_{J_{ij}}^\sigma$. Namely, let us denote 
    \begin{align*}
        \alpha_{\text{LSI}} = \min_{i\in \{0,1\}, j\in [n/6M],\sigma(J_{ij}^c)} \alpha_{\text{LSI}}( \Mnn (J_{ij}^\sigma))\,,
    \end{align*}
    i.e., the minimal log-Sobolev constant on a segment given the configuration on the complement and restricted to $\ell_0 \loc$. Then 
    \begin{align}\label{eq:approx-entropy-factorization}
        \text{Ent}(f) \le 2 \frac{1}{\alpha_{\text{LSI}}} \Big( \sum_{j} \bar \mu [ \bar{\mathcal E}_{J_{0j}}^\sigma (\sqrt{f},\sqrt{f})]   + \sum_j  \bar \mu [ \bar{\mathcal E}_{J_{1j}}^\sigma (\sqrt{f},\sqrt{f})]  \Big)\,,
    \end{align}
    where the expectations are over the $\sigma$ superscripts denoting the configurations on $J_{ij}^c$. 

        By Lemma~\ref{l:relabeling}, up to a relabeling of the particle labels, the chain on any of the segments $J_{ij}$ is equivalent to an $\eps$-positively biased adjacent transposition chain $\Mnn([4M])$ on $[4M]$ restricted to some $\vec{\ell}^{(4M)} \in \mathcal L_{4M}$. In particular, Corollary~\ref{cor:polynomial-LSI-and-mixing} and $p_{\min} \ge n^{-\kappa}$ can be applied to bound $1/\alpha_{\text{LSI}}$ by $C'(\ep) (4M)^{C_2} (\kappa \log n)$ which is $C'(\ep) (e^{O(1/\ep)}\log^3 n)^{C_2}$ for a $C_2$ that is $\eps$-independent.

    The part of the right-hand side in the parentheticals of~\eqref{eq:approx-entropy-factorization} is bounded by the Dirichlet form of the restricted adjacent transposition on the entire segment $[n]$ up to a factor of $2n$ because each possible transposition gets counted at most twice in the above right-hand side, and all transpositions that would be allowed in just a segment $J_{ij}$ would be allowed on the full segment. Altogether, we get the log-Sobolev inequality for the $\Mnn$ chain on $[n]$ restricted to $\ell_0\loc$ of 
    \begin{align*}
        \text{Ent}(f) \le C'(\eps) n (\log n)^{C_2} \cdot \bar {\mathcal E}(\sqrt{f},\sqrt{f})\,.
    \end{align*}
    %\medskip
    That is to say, the inverse log-Sobolev constant of the $\ell_0\loc$ restricted chain $\overline{\sigma}_t$ is at most $C'(\eps) n (\log n)^{C_2}$ for $C_2$ that is $\eps$-independent. 
    By the well-known relation between log-Sobolev constant and the mixing time that only incurs a $\log \log \frac{1}{\mu_{\min}}$ factor, and the lower bound of $\mu_{\min} \ge e^{ n^2 \log p_{\min}}$, the comparison factor is at most $(2+\kappa) \log n$, and we conclude.
    % 
    % Applying Proposition~\ref{prop:block-dynamics} we get the bound on the  inverse spectral gap of the $\ell_0\loc$-restricted adjacent transposition shuffle $\overline{\sigma}_t$ by 
    % \begin{align*}
    %     \gamma^{-1} \le 2 C'(\ep) C_1 n (\log n)^{3C_2}\,,
    % \end{align*}
    % where the factor $2$ is from $\chi$ and the factor $n$ comes from the tensorization in discrete time of the gap of $B_0, B_1$. 
        \end{proof}

\subsection{Proof of main theorem}

We can now combine Lemma~\ref{lem:fast-mixing-within-ell-loc} with the burn-in result for $\Mnn$ from Lemma~\ref{lem:single-site-burn-in}, and a simple lower bound, to deduce our main result.  

\begin{proof}[\textbf{\emph{Proof of Theorem~\ref{thm:main}}}]
    \emph{Upper bound}.   We first prove the result for vectors $p_{ij}$ satisfying $p_{\min} \ge n^{-\kappa}$, then use a coupling argument to allow $p_{\min}$ arbitrarily close to $0$.
    
    Let $(\bar \si_t)_{t\ge 0}$ be the adjacent transposition shuffle restricted to $\ell_0\loc$. Let $\hat \si_t$ be the chain that follows the unrestricted $\Mnn$ chain $\si_t$ for time $T_0:=C_\eps n^2$, and then starting from $\hat \si_{T_0} = \si_{T_0}$, follows the restricted chain $\hat \si_{T_0 + s} = \bar \si_s$ with $\bar \si_0 = \hat \si_{T_0}$, coupled so that $\hat \si_{T_0+s}$ agrees with $\si_{T_0+s}$ on the event that $(\si_{T_0+r})_{r\le s}$ is entirely in $\ell_0\loc$. 
    
    By this coupling, for $t = T_0 + s$ and $s\le n^3 - C_\ep n^2$, we get for any event $A \subseteq S_n$, 
    \begin{align*}
        \max_{x_0} \mathbb P_{x_0} (\si_t \in A) \le  \max_{x_0}  \mathbb P_{x_0}\Big(\bigcup_{r}\in [T_0 , n^3] \{\si_{r}\notin \ell_0\loc\}\Big) + \max_{y_0 \in \ell_0\loc} \mathbb P_{y_0}( \bar \si_{s} \in A)\,.
    \end{align*}
    By Lemma~\ref{lem:fast-mixing-within-ell-loc}, if $s\ge C(\ep) n (\log n)^{C_1+1}$ for $C(\ep) = e^{O(1/\ep^2)}$, $C_1=O(1/\ep)$, then the second term is at most 
    $$\mu(A \mid \ell_0\loc) + n^{-10}\le \mu(A) + 3n^{-10}\,,$$
    where we use Lemma~\ref{lem:stationary-highprob} to remove the conditioning. 
    By Lemma~\ref{lem:single-site-burn-in} and a union bound over ${r}\in [T_0,n^3]$, the first term is at most $n^{-7}$, yielding 
    \begin{align*}
        \max_{x_0} \mathbb P_{x_0} (\si_t \in A) \le \mu(A) + 4n^{-7}\,.
    \end{align*}
    The same bound holds for $A^c$. Hence for any $\de$, for large enough $n$, we achieve total variation distance of $\delta$ in time $T_0 + C(\ep) n(\log n)^{C_1+1}\log (1/\de)$, giving the upper bounds for the mixing time as well as for pre-cutoff. 

    It remains to extend to the case where  $p_{ij}$ can get arbitrarily close to $\{0,1\}$. Suppose we have any vector $\vec{p}$ in the setting of the theorem, and let $\vec{p}'$ be its truncation so that for any $i<j$ having $p_{ij} \ge 1-n^{-\kappa}$, one sets $p'_{ij}= 1-p_{ji}
'= 1-n^{-\kappa}$ (and leaves the other $p_{ij}$ unchanged). Let $\mu'$ be the stationary distribution of the latter. By stationarity and a triangle inequality, 
    \begin{align*}
       \max_{x_0} \| & \mathbb P_{x_0}(\sigma_t\in \cdot) - \mu  \|_{\tv} & \\ 
       & \le \max_{x_0}\| \mathbb P_{x_0}(\sigma_t\in \cdot)  - \mathbb P_{x_0}(\sigma_t' \in \cdot) \|_{\tv} + \| \mathbb P_\mu (\sigma_t \in \cdot)  - \mathbb P_{\mu}(\sigma_t'\in \cdot) \|_\tv \\ 
       & \qquad + \max_{x_0}\|\mathbb P_{x_0}(\sigma_t' \in \cdot) -\mu'\|_{\tv} + \|\mathbb P_{\mu}(\sigma_t'\in \cdot) - \mu'\|_{\tv}
    \end{align*}
    For $t \ge T_0 + C(\ep) n(\log n)^{C_1+1}\log (3/\de)= O(n^2)$ the third and fourth terms on the right are at most $\delta/3$ each, by the mixing time bound with swap rates $\vec{p}'$. The first two terms are bounded by $2t n^{-\kappa}$ because the two chains are trivially coupled in such a way that they take the same move in each step except with probability $n^{-\kappa}$. Since $\kappa >2$, this implies that at time $T_0 + C(\ep) n(\log n)^{C_1+1}\log (3/\de)$ the right-hand side is at most $\delta$ yielding the claimed upper bound.

\medskip
\noindent 
\emph{Lower bound}.  For the lower bound, we use a test statistic, which is the location of particle~$1$. Consider the initialization that is fully inverted, $\sigma_0 = (n, n-1, ..., 1)$. The location of particle~$1$, i.e., $\sigma^{-1}(1)$ is stochastically to the right of the process which flips a $\text{Ber}\prc{n-1}$ coin to decide whether or not to move, and if it comes up heads, then it deterministically moves to the left. It is straightforward to see by a Chernoff bound that for any $\eta>0$, this latter process takes time at least $(1-\eta) n^2$ to have $\Omega(1)$ probability of particle $1$ being in the interval $[\sqrt n]$. Therefore, by the stochastic domination, 
\begin{align*}
    \mathbb P(\sigma^{-1}_{(1-\eta)n^2}(1) \in [\sqrt n])  = o(1)\,.
\end{align*}
On the other hand, by Lemma~\ref{lem:stationary-highprob}, the event $\sigma^{-1}(1) \in [\sqrt n]$ has probability $1-o(1)$ under stationarity as long as $\fc{\sqrt{n}}{\log n} \gg 1/\eps$. This implies for every $\eta>0$, a lower bound of $\Tmix(1-\delta) \ge (1-\eta) n^2$ for all $\delta>0$ and all $n =\wt \Omega(1/\eps^2)$. 
\end{proof}

\subsection*{Acknowledgments}
The authors sincerely thank the anonymous referees for their careful reading and suggestions. This work originated from conversations at the Rocky Mountain Summer Workshop in Algorithms, Probability, and Combinatorics at Colorado State University. The research of R.G.\ is supported in part by NSF CAREER grant 2440509 and NSF DMS grant 2246780. The research of E.V.\ is supported by NSF grant CCF-2147094.

\bibliographystyle{alpha}
\bibliography{refs}

\end{document}